\documentclass{article}
\usepackage{amscd,amssymb,latexsym,amsmath,cite,amsthm}
\usepackage{epsfig}
\def\eq#1{{\rm(\ref{#1})}}
\theoremstyle{plain}
\newtheorem{thm}{Theorem}[section]
\newtheorem{prop}[thm]{Proposition}
\newtheorem{lem}[thm]{Lemma}

\newtheorem{quest}[thm]{Question}
\theoremstyle{definition}
\newtheorem{dfn}[thm]{Definition}
\newtheorem{ex}[thm]{Example}
\newtheorem{rem}[thm]{Remark}
\def\Re{\mathop{\rm Re}}

\def\Ric{\mathop{\rm Ric}}
\def\Ker{\mathop{\rm Ker}}

\def\U{\mathbin{\rm U}}
\def\u{\mathbin{\mathfrak u}}
\def\Vol{\mathop{\rm Vol}\nolimits}
\def\Exp{\mathop{\rm Exp}\nolimits}

\def\id{\mathop{\rm id}}
\def\ge{\geqslant}
\def\le{\leqslant}

\def\R{{\mathbin{\mathbb R}}}
\def\Z{{\mathbin{\mathbb Z}}}

\def\C{{\mathbin{\mathbb C}}}
\def\g{{\mathbin{\mathfrak g}}}
\def\ovB{\,\overline{\!B}}
\def\al{\alpha}
\def\be{\beta}
\def\na{\nabla}
\def\ga{\gamma}
\def\de{\delta}
\def\io{\iota}
\def\ep{\epsilon}
\def\ze{\zeta}
\def\la{\lambda}
\def\om{\omega}
\def\vp{\varphi}
\def\th{\theta}
\def\up{\upsilon}
\def\Ga{\Gamma}
\def\De{\Delta}
\def\Om{\Omega}

\def\Up{\Upsilon}
\def\pd{\partial}
\def\ts{\textstyle}

\def\w{\wedge}
\def\iy{\infty}
\def\lt{\ltimes}
\def\ra{\rightarrow}

\def\longra{\longrightarrow}
\def\t{\times}

\def\ci{\circ}
\def\ti{\tilde}
\def\d{{\rm d}}
\def\f{\frac}

\def\ot{\otimes}
\def\op{\oplus}

\def\ha{{\ts\frac{1}{2}}}
\def\bs{\boldsymbol}
\def\md#1{\vert #1 \vert}
\def\bmd#1{\big\vert #1 \big\vert}
\def\nm#1{\Vert #1 \Vert}
\def\bnm#1{\big\Vert #1 \big\Vert}
\def\ms#1{\vert #1 \vert^2}
\def\an#1{\langle #1 \rangle}
\def\ban#1{\bigl\langle #1 \bigr\rangle}
\begin{document}
\title{On the existence of Hamiltonian stationary Lagrangian
submanifolds in symplectic manifolds}
\author{Dominic Joyce, Yng-Ing Lee and Richard Schoen}
\date{}
\maketitle
\begin{abstract}
Let $(M,\om)$ be a compact symplectic $2n$-manifold, and $g$ a
Riemannian metric on $M$ compatible with $\om$. For instance, $g$
could be K\"ahler, with K\"ahler form $\om$. Consider compact
Lagrangian submanifolds $L$ of $M$. We call $L$ {\it Hamiltonian
stationary}, or {\it H-minimal}, if it is a critical point of the
volume functional $\Vol_g$ under Hamiltonian deformations, computing
$\Vol_g(L)$ using $g\vert_L$. It is called {\it Hamiltonian
stable\/} if in addition the second variation of $\Vol_g$ under
Hamiltonian deformations is nonnegative.

Our main result is that if $L$ is a compact, Hamiltonian stationary
Lagrangian in $\C^n$ which is {\it Hamiltonian rigid}, then for any
$M,\om,g$ as above there exist compact Hamiltonian stationary
Lagrangians $L'$ in $M$ contained in a small ball about some $p\in
M$ and locally modelled on $t L$ for small $t>0$, identifying $M$
near $p$ with $\C^n$ near $0$. If $L$ is Hamiltonian stable, we can
take $L'$ to be Hamiltonian stable.

Applying this to known examples $L$ in $\C^n$ shows that there exist
families of Hamiltonian stable, Hamiltonian stationary Lagrangians
diffeomorphic to $T^n$, and to $({\cal S}^1\t{\cal S}^{n-1})/\Z_2$,
and with other topologies, in every compact symplectic $2n$-manifold
$(M,\om)$ with compatible metric~$g$.
\end{abstract}

\section{Introduction}
\label{hs1}

Let $(M,\om)$ be a symplectic manifold of real dimension $2n$, and
$g$ a Riemannian metric on $M$ compatible with $\om$, and $J$ the
associated almost complex structure, so that $\om(v,w)=g(Jv,w)$ for
vector fields $v,w$ on $M$. For example, $J$ could be an integrable
complex structure, $g$ a K\"ahler metric on $(M,J)$, and $\om$ the
K\"ahler form. This paper concerns some special classes of compact
Lagrangian submanifolds $L$ in $(M,\om)$:
\begin{itemize}
\setlength{\itemsep}{0pt}
\setlength{\parsep}{0pt}
\item We call $L$ {\it Hamiltonian stationary\/} (or {\it
H-minimal\/}) if $L$ has stationary volume amongst Hamiltonian
equivalent Lagrangians $L'$. The Euler--Lagrange equation for
Hamiltonian stationary Lagrangians is $\d^*\al_H=0$, where $H$ is
the mean curvature vector on $L$, $\al_H$ the 1-form on $L$ defined
by $\al_H(\cdot)=\om(H,\cdot)$, and $\d^*$ the Hodge dual of the
exterior derivative~$\d$.
\item If $L$ is Hamiltonian stationary, we call $L$ ({\it
Hamiltonian\/}) {\it stable\/} if the second variation of volume at
$L$ amongst Hamiltonian equivalent Lagrangians $L'$ is nonnegative.
\end{itemize}
Now let $M$ be $\C^n$ with its Euclidean K\"ahler structure
$J_0,g_0,\om_0$, and $L$ be a compact Lagrangian in $\C^n$. Then
\begin{itemize}
\setlength{\itemsep}{0pt}
\setlength{\parsep}{0pt}
\item If $L$ is Hamiltonian stationary in $\C^n$, we call $L$ ({\it
Hamiltonian\/}) {\it rigid\/} if all infinitesimal Hamiltonian
deformations of $L$ as a Hamiltonian stationary Lagrangian are
induced by the action of the Lie algebra $\u(n)\op\C^n$ of the
automorphism group $\U(n)\lt\C^n$ of~$(\C^n,J_0,g_0)$.
\end{itemize}
More details are given in \S\ref{hs2}. Hamiltonian stationary
Lagrangians were defined and studied by Oh \cite{Oh1,Oh2}, in the
K\"ahler case. In \cite[Th.~IV]{Oh2} he proves that for
$a_1,\ldots,a_n>0$, the torus $T^n_{a_1,\ldots,a_n}$ in $\C^n$ given
by
\begin{equation}
T^n_{a_1,\ldots,a_n}=\bigl\{(z_1,\ldots,z_n)\in\C^n:\md{z_j}=a_j, \;
j=1,\ldots,n\bigr\}
\label{hs1eq1}
\end{equation}
is a stable, rigid, Hamiltonian stationary Lagrangian in $\C^n$. In
Example \ref{hs2ex2} we show that the Lagrangian $L_n$ diffeomorphic
to $({\cal S}^1\t{\cal S}^{n-1})/\Z_2$ given by
\begin{equation}
\ts L_n=\bigl\{(x_1 e^{is},\ldots,x_n e^{is}): 0\leq s <
\pi,\;\sum_{j=1}^n x_j^2=1,\; (x_1, \ldots, x_n)\in \R^n \bigr\},
\label{hs1eq2}
\end{equation}
is a stable, rigid, Hamiltonian stationary Lagrangian in $\C^n$, and
Example \ref{hs2ex4} gives more examples due to Amarzaya and
Ohnita~\cite{AmOh}.

Hamiltonian stationary Lagrangians are interesting as they can be
viewed as the `best' representatives of a Hamiltonian isotopy class
of Lagrangians, and therefore studying Hamiltonian stationary
Lagrangians may give us some understanding of the family of all
Lagrangians (see \S\ref{hs7} on this point). Also, for compact,
nonsingular, graded Lagrangians in a Calabi--Yau manifold, to be
Hamiltonian stationary is equivalent to being special Lagrangian.
Thus, Hamiltonian stationary Lagrangians are a generalization of
special Lagrangians.

The goal of this paper is to prove the following:
\medskip

\noindent{\bf Theorem A.} {\it Suppose\/ $(M,\om)$ is a compact
symplectic\/ $2n$-manifold, $g$ a Riemannian metric on $M$
compatible with\/ $\om,$ and\/ $L$ is a compact, Hamiltonian rigid,
Hamiltonian stationary Lagrangian in\/ $\C^n$. Then there exist
compact, Hamiltonian stationary Lagrangians\/ $L'$ in $M$ which are
diffeomorphic to\/ $L,$ such that\/ $L'$ is contained in a small
ball about some point $p\in M,$ and identifying $M$ near\/ $p$
with\/ $\C^n$ near\/ $0$ in geodesic normal coordinates, $L'$ is a
small deformation of\/ $t L$ for small\/ $t>0$. If\/ $L$ is also
Hamiltonian stable, we can take\/ $L'$ to be Hamiltonian stable.}
\medskip

Applying Theorem A with $L=T^n_{a_1,\ldots,a_n}$ or $L_n$ yields:
\medskip

\noindent{\bf Corollary B.} {\it Suppose\/ $(M,\om)$ is a compact
symplectic\/ $2n$-manifold, and\/ $g$ a Riemannian metric on $M$
compatible with\/ $\om$. Then there exist stable, Hamiltonian
stationary Lagrangians\/ $L'$ in\/ $M$ diffeomorphic to\/ $T^n$ and
to\/ $({\cal S}^1\t{\cal S}^{n-1})/\Z_2,$ which are locally modelled
on $tT^n_{a_1,\ldots,a_n}$ in\/ \eq{hs1eq1} or on $tL_n$ in\/
\eq{hs1eq2} for small\/ $t>0$ near some point\/ $p\in M,$
identifying $M$ near\/ $p$ with\/ $\C^n$ near\/~$0$.}
\medskip

We note that a special case of Corollary B has been proved
independently by Butscher and Corvino \cite{BuCo}, using a similar
method. They fix $n=2$, suppose $(M,J,g)$ is K\"ahler, take
$L=T^2_{a_1,a_2}$, and also assume a nondegeneracy condition on the
metric $g$ near the point $p$ which we do not need. One difference
between our approach and theirs is that they fix the point $p\in M$
where they glue in $L$ in advance and make assumptions about it,
whereas we show we can glue in $L$ near $p$ for some unknown
point~$p\in M$. A different nondegeneracy condition for K\"ahler
manifolds of any dimension and $L=T^n_{a_1,\ldots,a_n}$ is obtained
by the second author recently \cite{Lee}. Butscher and Corvino
believe that their method can also be generalized to higher
dimensions.

Another approach to the construction of Hamiltonian stationary
Lagrangians in general K\"ahler and symplectic manifolds is the
variational approach introduced by Schoen and Wolfson \cite{ScWo}.
The idea there is to minimize volume among Lagrangian cycles
representing a given homology (or homotopy) class to produce a
minimizer in a class of singular Lagrangian submanifolds. One then
hopes to study the regularity properties of these minimizing cycles.
This minimization can be done in general dimensions in the class of
Lagrangian integral currents, but the regularity theory is still
missing in general. For the two dimensional problem, one can
minimize in the class of surfaces which are images of a fixed
surface (under $W^{1,2}$ maps), and the paper \cite{ScWo} develops
the existence and regularity theory for this problem. It is shown
that such minimizers are smooth branched Hamiltonian stationary
surfaces outside a finite number of singular points at which the
possible tangent cones can be described.

We begin in \S\ref{hs2} with some background material from
symplectic geometry, the definition of Hamiltonian stationary,
Hamiltonian stable, and Hamiltonian rigid Lagrangians, and examples
of stable and rigid Hamiltonian stationary Lagrangians in $\C^n$.
Given a compact symplectic manifold $(M,\om)$ with compatible metric
Riemannian $g$, \S\ref{hs3} constructs a smooth family of Darboux
coordinate systems $\Up_{p,\up}:B_\ep\ra M$ for each $p\in M$ such
that the metric $\Up_{p,\up}^*(g)$ on $B_\ep\subset\C^n$ is close to
Euclidean metric $g_0$ on $B_\ep$ in $C^k$ for all $k\ge 0$,
uniformly in~$p\in M$.

Section \ref{hs4} sets up the notation for the proof of Theorem A,
recasting it as solving one of a family of fourth-order nonlinear
elliptic p.d.e.s $P^t_{p,\up}(f)=0$ on $L'$, for small $t>0$ and
$(p,\up)$ in the $\U(n)$-frame bundle $U$ of $M$, where for small
$t,f$ the equation $P^t_{p,\up}(f)=0$ approximates a fourth-order
linear elliptic equation ${\cal L}f=0$. Section \ref{hs5} proves
that we can solve this equation mod $\Ker{\cal L}$ for all $(p,\up)$
in $U$, and there is a unique solution $f^t_{p,\up}$ which is
$L^2$-orthogonal to $\Ker{\cal L}$ and small in $C^{4,\ga}$. Section
\ref{hs6} shows that $P^t_{p,\up}(f^t_{p,\up})=0$ if and only if
$(p,\up)$ is a stationary point of a smooth function $K^t$ on the
compact manifold $U$, and deduces Theorem A. Finally, \S\ref{hs7}
speculates on the possibility of defining invariants `counting'
Hamiltonian stationary Lagrangians in a fixed Hamiltonian isotopy
class in~$(M,\om)$.
\medskip

\noindent {\bf Acknowledgements:} Much progress on this project was
made while Yng-Ing visited Dominic in Oxford in July 2008. Yng-Ing
and Dominic would like to thank the Mathematical Institute, Oxford
University, and the Taida Institute for Mathematical Sciences,
National Taiwan University, for support which made this visit
possible. Yng-Ing also thanks Yoshihiro Ohnita for telling her of
his work with Amarzaya \cite{AmOh} on examples of stable and rigid
Hamiltonian stationary Lagrangians in~$\C^n$.

\section{Background material}
\label{hs2}

\subsection{Background from symplectic geometry}
\label{hs21}

We start by recalling some elementary symplectic geometry, which can
be found in McDuff and Salamon \cite{McSa}. Here are the basic
definitions.

\begin{dfn} Let $M$ be a smooth manifold of even dimension $2n$.
A closed $2$-form $\om$ on $M$ is called a {\it symplectic form} if
the $2n$-form $\om^n$ is nonzero at every point of $M$. Then
$(M,\om)$ is called a {\it symplectic manifold}. A submanifold $L$
in $M$ is called {\it Lagrangian} if $\dim L=n=\ha\dim M$
and~$\om\vert_L\equiv 0$.
\label{hs2def1}
\end{dfn}

The simplest example of a symplectic manifold is~$\C^n$.

\begin{ex} Let $\C^n$ have complex coordinates $(z_1,\ldots,z_n)$,
where $z_j=x_j+iy_j$ with $i=\sqrt{-1}$. Define the standard
Euclidean metric $g_0$, symplectic form $\om_0$, and complex
structure $J_0$ on $\C^n$ by
\begin{equation}
\begin{split}
g_0&=\ts\sum_{j=1}^n\ms{\d z_j}=\sum_{j=1}^n(\d x_j^2+\d y_j^2),\\
\om_0&=\ts\frac{i}{2}\sum_{j=1}^n\d z_j\w\d\bar{z}_j=\sum_{j=1}^n\d
x_j\w\d y_j,\quad\text{and}\\
J_0&=\ts\sum_{j=1}^n\bigl(i\d z_j \ot \frac{\pd}{\pd
z_j}-i\d\bar{z}_j\ot \frac{\pd}{\pd\bar{z}_j}\bigr)
=\sum_{j=1}^n\bigl(\d x_j\ot\frac{\pd}{\pd y_j}- \d
y_j\ot\frac{\pd}{\pd x_j}\bigr),
\end{split}
\label{hs2eq1}
\end{equation}
noting that $\d z_j=\d x_j+i\d y_j$ and $\frac{\pd}{\pd
z_j}=\ha\bigl(\frac{\pd}{\pd x_j}-i\frac{\pd}{\pd y_j}\bigr)$. Then
$(\C^n,\om_0)$ is a symplectic manifold, and $g_0$ is a K\"ahler
metric on $(\C^{2n},J)$ with K\"ahler form $\om_0$. For $\ep>0$,
write $B_\ep$ for the open ball of radius $\ep$ about 0 in~$\C^n$.
\label{hs2ex1}
\end{ex}

{\it Darboux' Theorem} \cite[Th.~3.15]{McSa} says that every
symplectic manifold is locally isomorphic to $(\C^n,\om_0)$.

\begin{thm} Let\/ $(M,\om)$ be a symplectic $2n$-manifold and\/
$p\in M$. Then there exist\/ $\ep>0$ and an embedding $\Up:B_\ep\ra
M$ with\/ $\Up(0)=p$ such that\/ $\Up^*(\om)=\om_0,$ where $\om_0$
is the standard symplectic form on~$\C^n\supset B_\ep$.
\label{hs2thm1}
\end{thm}

Let $L$ be a real $n$-manifold. Then its tangent bundle $T^*L$ has a
{\it canonical symplectic form} $\hat\om$, defined as follows. Let
$(x_1,\ldots,x_n)$ be local coordinates on $L$. Extend them to local
coordinates $(x_1,\ldots,x_n,y_1,\ldots,y_n)$ on $T^*L$ such that
$(x_1,\ldots,y_n)$ represents the 1-form $y_1\d x_1+\cdots+y_n \d
x_n$ in $T_{(x_1,\ldots,x_n)}^*L$. Then $\hat\om=\d x_1\w\d y_1+
\cdots+\d x_n\w\d y_n$. Identify $L$ with the zero section in
$T^*L$. Then $L$ is a {\it Lagrangian submanifold\/} of
$(T^*L,\hat\om)$. The {\it Lagrangian Neighbourhood Theorem}
\cite[Th.~3.33]{McSa} shows that any compact Lagrangian submanifold
$L$ in a symplectic manifold looks locally like the zero section
in~$T^*L$.

\begin{thm} Let\/ $(M,\om)$ be a symplectic manifold and\/
$L\subset M$ a compact Lagrangian submanifold. Then there exists an
open tubular neighbourhood\/ $T$ of the zero section $L$ in $T^*L,$
and an embedding $\Phi:T\ra M$ with\/ $\Phi\vert_L=\id:L\ra L$ and\/
$\Phi^*(\om)=\hat\om,$ where $\hat\om$ is the canonical symplectic
structure on~$T^*L$.
\label{hs2thm2}
\end{thm}

We shall call $T,\Phi$ a {\it Lagrangian neighbourhood\/} of $L$.
Such neighbourhoods are useful for parametrizing nearby Lagrangian
submanifolds of $M$. Suppose that $\ti L$ is a Lagrangian
submanifold of $M$ which is $C^1$-close to $L$. Then $\ti L$ lies in
$\Phi(T)$, and is the image $\Phi(\Ga_\al)$ of the graph $\Ga_\al$
of a unique $C^1$-small 1-form $\al$ on $L$. As $\ti L$ is
Lagrangian and $\Phi^*(\om)=\hat\om$ we see that
$\hat\om\vert_{\Ga_\al}\equiv 0$. But $\hat\om\vert_{\Ga_\al}
=-\pi^*(\d\al)$, where $\pi:\Ga_\al\ra L$ is the natural projection.
Hence $\d\al=0$, and $\al$ is a {\it closed\/ $1$-form}. This
establishes a 1-1 correspondence between $C^1$-small closed 1-forms
on $L$ and Lagrangian submanifolds $\ti L$ close to $L$ in~$M$.

Let $(M,\om)$ be a compact symplectic manifold and $F:M\ra\R$ a
smooth function. The {\it Hamiltonian vector field\/} $v_F$ of $F$
is the unique vector field satisfying $v_F\cdot\om=\d F$. The Lie
derivative satisfies ${\cal L}_{v_F}\om=v_F\cdot\d\om
+\d(v_F\cdot\om)=0$, so the trajectory of $v_F$ gives a 1-parameter
family of symplectomorphisms $\Exp (sv_F):M\ra M$ for $s\in\R$,
called the {\it Hamiltonian flow\/} of $F$. If $L$ is a compact
Lagrangian in $M$ then $\Exp(sv_F)L$ is also a compact Lagrangian
in~$M$.

Two compact Lagrangians $L,L'$ in $(M,\om)$ are called {\it
Hamiltonian equivalent\/} if there exist Lagrangians
$L=L_0,\ldots,L_k=L'$ and functions $F_1,\ldots,F_k$ on $M$ with
$L_j=\Exp(v_{F_j})L_{j-1}$ for $j=1,\ldots,k$. In the situation of
Theorem \ref{hs2thm2}, a Lagrangian $\ti L$ which is $C^1$ close to
$L$ is Hamiltonian equivalent to $L$ if it corresponds to the graph
$\Ga_{\d f}$ of an {\it exact\/} 1-form $\d f$ on~$L$.

\subsection{The volume functional on Lagrangians}
\label{hs22}

Now let $(M,\om)$ be a symplectic manifold and $g$ a Riemannian
metric on $M$ compatible with $\om$. For a compact Lagrangian $L$ in
$M$, write $\Vol_g(L)$ for the volume of $L$, computed using the
induced Riemannian metric $h=g\vert_L$. The {\it volume
functional\/} is $\Vol_g:L\mapsto \Vol_g(L)$, regarded as a
functional on the infinite-dimensional manifold of all compact
Lagrangians in~$(M,\om)$.

\begin{dfn} A compact Lagrangian submanifold $L$ in $M$ is called
{\it Hamiltonian stationary}, or {\it H-minimal}, if it is a
critical point of the volume functional on its Hamiltonian
equivalence class of Lagrangians. That is, $L$ is Hamiltonian
stationary if
\begin{equation}
\ts\frac{\d}{\d s}\Vol_g\bigl(\Exp(sv_F)L\bigr)\big\vert_{s=0}=0
\label{hs2eq2}
\end{equation}
for all smooth $F:M\ra\R$. By Oh \cite[Th.~I]{Oh2}, who attributes
the result to Weinstein, \eq{hs2eq2} is equivalent to the
Euler--Lagrange equation
\begin{equation}
\d^*\al_H=0,
\label{hs2eq3}
\end{equation}
where $H$ is the mean curvature vector of $L$, and $\al_H=
(H\cdot\om)\vert_L$ is the associated 1-form on $L$, and $\d^*$ is
the Hodge dual of the exterior derivative $\d$ on $L$, computed
using the metric $h=g\vert_L$. Oh assumed that $M$ is a K\"ahler
manifold in his paper. However, the proof for \eq{hs2eq3} also works
when $M$ is a symplectic manifold with a compatible metric.
Explicitly, if $(x_1,\ldots,x_n)$ are local coordinates on $L$ and
$h=h_{ab}\,\d x_a\d x_b$, $\al=\al_a \d x_a$, we have
\begin{equation}
\d^*\al=-\frac{\pd h^{ab}}{\pd x_b}\,\al_a-h^{ab}\,\frac{\pd
\al_a}{\pd x_b}-\ha \, h^{ab}\al_a\frac{\pd}{\pd
x_b}\bigl(\ln\det(h_{cd})\bigr).
\label{hs2eq4}
\end{equation}
Here we use the convention that repeated indices stand for a
summation whenever there is no confusion.
\label{hs2def2}
\end{dfn}

\begin{dfn} Following Oh \cite[Def.~2.5]{Oh1}, a compact
Hamiltonian stationary Lagrangian $L$ in $M$ is called {\it
Hamiltonian stable}, or just {\it stable}, if the second variation
of $\Vol_g$ is nonnegative for all Hamiltonian variations of $L$,
that is, if
\begin{equation}
\ts\frac{\d^2}{\d s^2}\Vol_g\bigl(\Exp(sv_F)L\bigr)\big\vert_{s=0}
\ge 0
\label{hs2eq5}
\end{equation}
for all smooth $F:M\ra\R$. A geometric expression for the left hand
side of \eq{hs2eq5} is computed by Oh \cite[\S 3]{Oh2} when $M$ is a
K\"ahler manifold. Define a tensor $S=S_{jkl}\d x_j\d x_k\d x_l$ on
$L$ by $S(u,v,w)=\ban{J(B(u,v)),w}$ for all vector fields $u,v,w$ on
$L$, where $B$ is the second fundamental form of $L$ in $M$, so that
$B(u,v)$ is a normal vector field to $L$ in $M$ and $J(B(u,v))$ is a
vector field on $L$. Then $S$ is symmetric by \cite[Lem.~3.1]{Oh2}.
Let $F:M\ra\R$ be smooth and $f=F\vert_L$. Then Oh
\cite[Th.~3.4]{Oh2} proves:
\begin{equation}
\frac{\d^2}{\d s^2}\Vol_g\bigl(\Exp(sv_F)L\bigr)
\big\vert_{s=0}=\int_L
\begin{aligned}[t]\bigl(&\an{\De_H\d f,\d f}-\Ric(J\d f,J\d f)\\
&-2\an{\d f\ot\d f\ot\al_H,S}+\an{\d f,\al_H}^2\bigr)\,\d V,
\end{aligned}
\label{hs2eq6}
\end{equation}
where $\De_H=\d\d^*+\d^*\d$ is the Hodge Laplacian, and $\Ric$ is
the Ricci curvature of $g$ on $M$.

By a similar computation, one can derive the second variation
formula of volume for Hamiltonian deformation at a Lagrangian $L$
(not necessarily Hamiltonian stationary) in a K\"ahler manifold.
From this expression, one can write down the linearized operator of
$-\d^*\al_H$. More precisely, given a smooth function $f$ on $L$, we
extend it to a smooth function $F$ on $M$ and consider $L_s=\Exp
(sv_F)L$ whose mean curvature vector is denoted by $H_s$. Then we
have
\begin{equation}
\begin{aligned}
{\cal L}f=-\f{\d}{\d s}\,(\d^*\al_{H_s})\big\vert_{s=0}=
\begin{aligned}[t]
\De^2 f+\d^*\al_{\Ric^{\perp}(J\,\na f)}&-2\d^* \al_{B(JH,\na f)}\\
&-JH(JH(f)),
\end{aligned}
\end{aligned}
\label{hs2eq7}
\end{equation}
which is a fourth-order linear elliptic operator from $C^\iy(L)$ to
$C^\iy(L)$. Here $\De f=\d^*\d f$ is the Hodge Laplacian on
functions, and the normal vector $\Ric^{\perp}(v)$ for a normal
vector $v$ is characterized by $\Ric(v,w)=\an{\Ric^{\perp}(v),w}$
for any normal vector~$w$.

When $M$ is just a symplectic manifold with a compatible metric, the
second variation formula of volume and the linearized operator
${\cal L}$ do not have such nice expressions. However, since we will
work on small balls in Darboux coordinates, the linearized operator
${\cal L}$ at $L$ will be very close to the corresponding linearized
operator at $L$ in $\C^n$. This is made more precise in Proposition
\ref{hs4prop}. The estimate is good enough to pursue our argument
and prove the theorems. The expression \eq{hs2eq7} for ${\cal L}$ at
a Lagrangian in a K\"ahler manifold is helpful in understanding the
general symplectic picture.

When $L$ is a compact Hamiltonian stationary Lagrangian in a
K\"ahler manifold, we have
\begin{equation}
\ts\frac{\d^2}{\d s^2}\Vol_g\bigl(\Exp(sv_F)L\bigr)\big\vert_{s=0}
=\ban{{\cal L}f,f}_{L^2(L)}.
\label{hs2eq8}
\end{equation}
Thus $L$ is Hamiltonian stable if and only if the eigenvalues of
${\cal L}$ are all nonnegative. The interpretation of the kernel
$\Ker{\cal L}$ will be important below. From \eq{hs2eq8} we can see
that if $f\in C^\iy(L)$ and $v$ is the normal vector field to $L$ in
$M$ with $\al_v=\d f$, then $v$ is an infinitesimal Hamiltonian
deformation of $L$ as a Hamiltonian stationary Lagrangian, if and
only if ${\cal L}f=0$.
\label{hs2def3}
\end{dfn}

Now suppose that $(M,J,g)$ is a {\it Calabi--Yau $n$-fold}. Then we
can choose a holomorphic $(n,0)$-form $\Om$ on $M$ with
$\nabla\Om=0$, normalized so that
\begin{equation*}
\om^n/n!=(-1)^{n(n-1)/2}(i/2)^n\Om\w\bar\Om.
\end{equation*}
If $L$ is an oriented Lagrangian in $M$, then $\Om\vert_L\equiv
e^{i\th}\d V_{g_L}$, where $\d V_{g_L}$ is the volume form of $L$
defined using the metric $g\vert_L$ and the orientation, and
$e^{i\th}:L\ra\U(1)=\{z\in\C:\md{z}=1\}$ is the {\it phase
function\/} of $L$. We call $L$ {\it special Lagrangian\/} if it has
constant phase.

In the Calabi--Yau case, the picture above simplifies in two ways.
Firstly, as $g$ is Ricci-flat, the Ricci curvature terms in
\eq{hs2eq6} and \eq{hs2eq7} are zero. Secondly, the 1-form $\al_H$
associated to the mean curvature $H$ of $L$ is given by
$\al_H=-\d\th$. (This does not imply that $\al_H$ is exact, as $\th$
maps $L\ra\R/2\pi\Z$ rather than $L\ra\R$.) Thus, the condition
\eq{hs2eq3} that $L$ be Hamiltonian stationary becomes
$\d^*\d\th=0$, that is, the phase function $e^{i\th}:L\ra\U(1)$ is
harmonic as a map into $\U(1)=\R/2\pi\Z$, though it is not harmonic
as a map into~$\C$.

If $e^{i\th}$ lifts continuously to $\th:L\ra\R$, that is, if $L$ is
{\it graded}, and also $L$ is compact, then the maximum principle
implies that $\th$ is constant, so $L$ is special Lagrangian. Hence,
any compact, graded, Hamiltonian stationary Lagrangian in a
Calabi--Yau $n$-fold is special Lagrangian. Also, $e^{i\th}$ lifts
continuously to $\th:L\ra\R$ if and only if $-[\al_H]=[\d\th]=0$
in~$H^1(L;\R)$.

\subsection{Rigid Hamiltonian stationary Lagrangians in $\C^n$}
\label{hs23}

We now discuss Lagrangians in $\C^n$, with $g_0,\om_0,J_0$ as in
Example \ref{hs2ex1}. This is Calabi--Yau, so as in \S\ref{hs22}, a
compact Hamiltonian stationary Lagrangian $L$ in $\C^n$ is special
Lagrangian if $-[\al_H]=[\d\th]=0$ in $H^1(L;\R)$. But there are no
compact special Lagrangians $L$ in $\C^n$ for $n>0$, as a
(co)homological calculation shows that $\Vol_{g_0}(L)=0$. Hence any
immersed, compact, Hamiltonian stationary Lagrangian $L$ in $\C^n$
has $[\d\th]\ne 0$ in $H^1(L;\R)$, so that $H^1(L;\R)\ne 0$. This
constrains the possible topologies of Hamiltonian stationary
Lagrangians in~$\C^n$.

We define some notation.

\begin{dfn} The Lie group $\U(n)\lt\C^n$ acts on $\C^n$
preserving $g_0,\om_0,J_0$. For $x$ in the Lie algebra
$\u(n)\op\C^n$, write $v_x$ for the vector field on $\C^n$ induced
by the action of $\U(n)\lt\C^n$ on $\C^n$, and let $\mu_x:\C^n\ra\R$
be a moment map for $v_x$, that is, $\d\mu_x=v_x\cdot\om$. Each such
moment map is a real quadratic polynomial on $\C^n$ whose
homogeneous quadratic part is of type $(1,1)$. Define $W_n$ to be
the vector space of such moment maps, that is, elements $Q$ of $W_n$
are of the form
\begin{equation*}
Q(z_1,\ldots,z_n)=\ts a+\sum_{j=1}^n(b_jz_j+\bar{b}_j\bar{z}_j)
+\sum_{j,k=1}^nc_{jk}z_j\bar{z}_k
\end{equation*}
for $a\in\R$ and $b_j,c_{jk}\in\C$ with $\bar{c}_{jk}=c_{kj}$. Then
$\dim W_n=n^2+2n+1$, which is $\dim(\u(n)\op\C^n)+1$, since as
moment maps are unique up to the addition of constants we have
$W_n/\an{1}\cong\u(n)\op\C^n$.
\label{hs2def4}
\end{dfn}

\begin{lem} Let\/ $L$ be a compact Hamiltonian stationary
Lagrangian in $\C^n,$ and\/ ${\cal L}:C^\iy(L)\ra C^\iy(L)$ be as
in\/ \eq{hs2eq7}. Then
\begin{equation}
\{Q\vert_L:Q\in W_n\}\subseteq\Ker{\cal L}.
\label{hs2eq9}
\end{equation}
If also\/ $L$ is connected then\/ $\dim\{Q\vert_L:Q\in
W_n\}=n^2+2n+1-\dim G,$ where $G$ is the Lie subgroup of\/
$\U(n)\lt\C^n$ preserving~$L$.
\label{hs2lem1}
\end{lem}

\begin{proof} As in \S\ref{hs22}, we have $f\in\Ker{\cal L}$ if and
only if $\d f$ is the 1-form associated to an infinitesimal
Hamiltonian deformation of $L$ as a Hamiltonian stationary
Lagrangian. If $x$ lies in $\u(n)\op\C^n$ with vector field $v_x$
and moment map $\mu_x$ then $v_x$ is the Hamiltonian vector field of
$\mu_x$, so flow by $v_x$ induces a Hamiltonian deformation of $L$.
Since the action of $\U(n)\lt\C^n$ takes Hamiltonian stationary
Lagrangians to Hamiltonian stationary Lagrangians, $v_x\vert_L$ is
an infinitesimal Hamiltonian deformation of $L$ as a Hamiltonian
stationary Lagrangian. The associated 1-form on $L$ is
$\d\mu_x\vert_L$. So $\mu_x\vert_L$ lies in $\Ker{\cal L}$. As every
$Q\in W_n$ is a moment map $\mu_x$ for some $x\in\u(n)\op\C^n$,
equation \eq{hs2eq9} follows.

For the second part, we have $\dim\{Q\vert_L:Q\in W_n\}=\dim
W_n-\dim\{Q\in W_n:Q\vert_L\equiv 0\}$, where $\dim W_n=n^2+2n+1$.
Let $\g$ be the Lie algebra of $G$, and let $x\in\g$ and $\mu_x$ be
a moment map for $x$. Then $v_x\vert_L$ is tangent to $L$, as $x$
preserves $L$, so $\d(\mu_x\vert_L)=v_x\cdot(\om \vert_L)=0$ as $L$
is Lagrangian. Since $L$ is connected this implies that
$\mu_x\vert_L$ is constant, and there is a unique choice of moment
map $\mu_x$ for $x$ such that $\mu_x\vert_L\equiv 0$. This yields an
isomorphism between $\g$ and $\{Q\in W_n:Q\vert_L\equiv 0\}$, so
$\dim\{Q\in W_n:Q\vert_L\equiv 0\}=\dim\g=\dim G$, and the lemma
follows.
\end{proof}

The following definition is new, as far as the authors know.

\begin{dfn} Let $L$ be a compact, Hamiltonian stationary
Lagrangian in $\C^n$. We call $L$ {\it Hamiltonian rigid}, or just
{\it rigid}, if equality holds in \eq{hs2eq9}. That is, $L$ is
Hamiltonian rigid if all infinitesimal Hamiltonian deformations of
$L$ as a Hamiltonian stationary Lagrangian in $\C^n$ come from the
action of $\u(n)\op\C^n$. By Lemma \ref{hs2lem1}, $L$ is Hamiltonian
rigid if $\dim\Ker{\cal L}=n^2+2n+1-\dim G$.
\label{hs2def5}
\end{dfn}

Note that any Hamiltonian rigid $L$ must be {\it connected}, since
otherwise we could apply different elements of $\u(n)\op\C^n$ to
different connected components of $L$ to prove equality does not
hold in \eq{hs2eq9}. It seems likely that in some sense, {\it
most\/} compact, connected, Hamiltonian stationary Lagrangians in
$\C^n$ are rigid, since in generic situations one expects the
kernels of elliptic operators to be as small as possible, given any
geometric constraints on index, etc.

We now give examples of stable, rigid, Hamiltonian stationary
Lagrangians in $\C^n$. In the first, for completeness, we give a
full proof of rigidity and stability.

\begin{ex} Define a submanifold $L_n$ in $\C^n$ by
\begin{equation*}
\ts L_n=\bigl\{(x_1 e^{is},\ldots,x_n e^{is}): 0\leq s <
\pi,\;\sum_{j=1}^n x_j^2=1,\; (x_1, \ldots, x_n)\in \R^n \bigr\},
\end{equation*}
which is diffeomorphic to $({\cal S}^1\t{\cal S}^{n-1})/\Z_2$ by
identifying $(s,x)$ and $(s+\pi,-x)$ in ${\cal S}^1 \t {\cal
S}^{n-1}.$ Lee and Wang \cite{LeWa} prove that $L_n$ is a
Hamiltonian stationary Lagrangian and its mean curvature vector
satisfies $H=-nF,$ where $F$ is the position vector. Moreover, the
induced metric on $L_n$ computed there is a product metric. More
precisely, assume that $\{v_j\}_{j=1}^{n-1}$ is a local orthonormal
basis for ${\cal S}^{n-1}$ and $v_0=(x_1,\ldots,x_n)$, then
$e_0=\f{\pd}{\pd s}=i\,e^{is}v_0=JF$ and $e_j=e^{is}v_j$,
$j=1,\ldots, n-1$, will form a local orthonormal basis for $L_n$. We
will prove that $L_n$ is Hamiltonian stable and rigid.

The linearized operator \eq{hs2eq7} for $L_n$ in $\C^n$ and $f \in
C^\iy(L_n)$ has the form
\begin{equation*}
{\cal L}f=\De^2 f-2\d^* \al_{B(JH,\na f)}-JH(JH(f)).
\end{equation*}
Note that $JH=-ne_0=-n\f{\pd}{\pd s}$ and
\begin{equation*}
\an{B(e_i,e_j),H}=-n\an{B(e_i,e_j), F}=n\de_{ij}.
\end{equation*}
Hence
\begin{equation*}
\an{B(JH, \na f), Je_j}=-\an{B(e_j, \na f), H}=-n\an{e_j, \na
f}=-ne_j(f),
\end{equation*}
that is, $B(JH, \na f)=-nJ\na f$ and $\al_{B(JH,\na f)}=n\d f$.
Therefore, we have
\begin{equation}
{\cal L}f=\De^2 f-2n\De f-n^2 \f{\pd^2
f}{\pd s^2}.
\label{hs2eq10}
\end{equation}
We can lift a function on $L_n$ to ${\cal S}^1 \t {\cal S}^{n-1}$
and consider $f$ as a $\Z_2$-invariant function on ${\cal S}^1 \t
{\cal S}^{n-1}$ instead. Since the induced metric is a product
metric, the products of eigenfunctions on ${\cal S}^1$ and
eigenfunctions on ${\cal S}^{n-1}$ respectively form a complete
basis for functions on ${\cal S}^1 \t {\cal S}^{n-1}$. That is,
$f=\sum_{k,l} a_{kl}\cos ks \,\vp _l+b_{kl}\sin ks \,\vp _l,$ where
$a_{kl}$ and $b_{kl}$ are constants, $k$ is a nonnegative integer,
and $\vp_l$ is an eigenfunction of the Laplacian on ${\cal S}^{n-1}$
with eigenvalue $\la _l$. Since the eigenfunctions of the Laplacian
on ${\cal S}^{n-1}$ are homogenous polynomials in $\R^n$, and $f$ is
$\Z_2$-invariant, the sum of $k$ and the degree $\deg\vp_l$ of
$\vp_l$ must be even.

From \eq{hs2eq10}, it follows that $\cos ks \,\vp _l$ and $\sin ks
\,\vp _l$ are eigenfunctions for ${\cal L}$ and form a complete
basis. To study $\Ker{\cal L}$, we only need to check these
functions. Rewrite \eq{hs2eq10} as ${\cal L}f=(\De -n)^2
f+n^2(-\f{\pd^2 }{\pd s^2}-1)f.$ Suppose $\cos ks \,\vp _l$ or $\sin
ks \,\vp _l$ is in $\Ker{\cal L}$. It follows that
\begin{equation*}
(k^2+\la_l-n)^2+n^2 (k^2-1)=0,
\end{equation*}
which implies $k\le 1$. When $k=1$, we must have $\la _l=n-1$ and
thus $\deg\vp_l$ is 1. These solutions and their combinations come
from the restriction functions on $L_n$ of elements in $W_n$ with
the form $\sum_{j=1}^n b_jz_j+\bar{b}_j\bar{z}_j$ in Definition
\ref{hs2def4}. When $k=0$, we must have $\la_l(\la_l-2n)=0$ and
thus $\la_l=0$ or $\la_l=2n$. Hence $\deg\vp_l= 0$ or 2. These
solutions and their combinations come from the restriction
functions of elements in $W_n$ with the form
$a+\sum_{j,k=1}^nc_{jk}z_j\bar{z}_k$ in Definition \ref{hs2def4}.
This completes the proof that $L_n$ is Hamiltonian rigid.

Now we show that $L_n$ is Hamiltonian stable. From \eq{hs2eq8},
this is equivalent to the eigenvalues for ${\cal L}$ all being
nonnegative. The eigenfunctions for ${\cal L}$ are $\cos ks\,
\vp_l$ and $\sin ks\,\vp_l$ with eigenvalue $(k^2+\la_l-n)^2+n^2
(k^2-1)$, which is nonnegative for $k\ge 1$. It is also
nonnegative when $k=0$ and $\deg\vp_l=0$ or $\deg\vp_l>1$. The
case $k=0$ and $\deg\vp_l=1$ does not occur as it is not
$\Z_2$-invariant. Therefore, $L_n$ is Hamiltonian stable.
\label{hs2ex2}
\end{ex}

\begin{ex} Let $a_1,\ldots,a_n>0$. Define a torus
$T^n_{a_1,\ldots,a_n}$ in $\C^n$ by
\begin{equation}
T^n_{a_1,\ldots,a_n}=\bigl\{(z_1,\ldots,z_n)\in\C^n:\md{z_j}=a_j, \;
j=1,\ldots,n\bigr\}.
\label{hs2eq11}
\end{equation}
Then Oh \cite[Th.~IV]{Oh2} proves that $T^n_{a_1,\ldots,a_n}$ is a
Hamiltonian stationary Lagrangian, and is stable and rigid. He also
remarks that two tori $T^n_{a_1,\ldots,a_n}$ and
$T^n_{a_1',\ldots,a_n'}$ are not Hamiltonian isotopic to one another
if $(a_1,\ldots,a_n)\ne (a_1',\ldots,a_n')$ (this needs caution: if
$a_1',\ldots,a_n'$ are a permutation of $a_1,\ldots,a_n$, then
$T^n_{a_1',\ldots,a_n'}$ is Hamiltonian isotopic to
$T^n_{a_1,\ldots,a_n}$, but after a diffeomorphism of $T^n$ not
isotopic to the identity), and conjectures that the
$T^n_{a_1,\ldots,a_n}$ are globally volume-minimizing under
Hamiltonian deformations.
\label{hs2ex3}
\end{ex}

\begin{ex} Amarzaya and Ohnita study Lagrangians with parallel
second fundamental form in \cite{AmOh}. These examples must be
Hamiltonian stationary since their mean curvature vectors are
parallel. They prove in the paper that the following irreducible
symmetric R-spaces are Hamiltonian stable and rigid:
\begin{itemize}
\setlength{\itemsep}{0pt}
\setlength{\parsep}{0pt}
\item[{\rm(i)}] $Q_{2,p+1}(\R)\cong ({\cal S}^1\t{\cal S}^{p+2})/\Z_2
\subset \C^{p+3}$ for $p\ge 1;$
\item[{\rm(ii)}] $\U(p)\subset \C^{p^2}$ for $p\ge 2;$
\item[{\rm(iii)}] $\U(p)/{\rm O}(p)\subset\C^{\f{p(p+1)}{2}}$ for
$p\ge 3;$
\item[{\rm(iv)}] $\U(2p)/{\rm Sp}(p)\subset \C^{p(2p-1)}$ for
$p\ge 3;$ and
\item[{\rm(v)}] $T\cdot(E_6/F_4)\subset\C^{27}$.
\end{itemize}
We refer to \cite{AmOh} for the details of these examples. Example
\ref{hs2ex2} is the same as $Q_{2,p+1}(\R)$ in (i), but their proof
of stability and rigidity is different to ours.
\label{hs2ex4}
\end{ex}

\section{Darboux coordinates with estimates}
\label{hs3}

\subsection{Families of Darboux coordinate systems for all $p\in M$}
\label{hs31}

We will need the following notation.

\begin{dfn} Let $(M,\om)$ be a symplectic $2n$-manifold, and $g$ a
Riemannian metric on $M$ compatible with $\om$. Then at each point
$p\in M$ the structures $\om\vert_p$, $g\vert_p$ on $T_pM$ are
isomorphic to $\om_0,g_0$ on $\C^n$, where $\om_0,g_0$ are as in
\eq{hs2eq1}. The linear automorphism group of $(\C^n,\om_0,g_0)$ is
the unitary group~$\U(n)$.

Write $U$ for the $\U(n)$ {\it frame bundle} of $M$, with projection
$\pi:U\ra M$. That is, points of $U$ are pairs $(p,\up)$ with $p\in
M$ and $\up:\C^n\ra T_pM$ an isomorphism of real vector spaces with
$\up^*(\om\vert_p)=\om_0$, and $\up^*(g\vert_p)=g_0$, and
$\pi:(p,\up)\mapsto p$. Then $\U(n)$ acts freely on the right on $U$
by $\ga:(p,\up)\mapsto (p,\up\ci\ga)$ for $\ga\in\U(n)$ and
$(p,\up)\in U$. The orbits of $\U(n)$ in $U$ are fibres of $\pi$,
and $\pi:U\ra M$ is a {\it principal $\U(n)$-bundle}. Thus $U$ is a
real manifold of dimension $n^2+2n$, which is compact if $M$ is
compact.
\label{hs3def1}
\end{dfn}

We now show that for compact $M$ we can choose Darboux coordinate
systems as in Theorem \ref{hs2thm1} for all $(p,\up)\in U$, smoothly
and $\U(n)$-equivariantly in~$(p,\up)$.

\begin{prop} Let\/ $(M,\om)$ be a compact symplectic $2n$-manifold,
$g$ a Riemannian metric on $M$ compatible with\/ $\om,$ and\/ $U$
the $\U(n)$ frame bundle of\/ $M$. Then for small\/ $\ep>0$ we can
choose a family of embeddings $\Up_{p,\up}:B_\ep\ra M$ depending
smoothly on $(p,\up)\in U,$ where $B_\ep$ is the ball of radius
$\ep$ about\/ $0$ in $\C^n,$ such that for all\/ $(p,\up)\in U$ we
have:
\begin{itemize}
\setlength{\itemsep}{0pt}
\setlength{\parsep}{0pt}
\item[{\rm(i)}] $\Up_{p,\up}(0)=p$ and\/ $\d\Up_{p,\up}\vert_0=
\up:\C^n\ra T_pM;$
\item[{\rm(ii)}] $\Up_{p,\up\ci\ga}\equiv\Up_{p,\up}\ci\ga$ for
all\/ $\ga\in\U(n);$
\item[{\rm(iii)}] $\Up_{p,\up}^*(\om)=\om_0=\ts\frac{i}{2}
\sum_{j=1}^n\d z_j\w\d\bar{z}_j;$ and
\item[{\rm(iv)}] $\Up_{p,\up}^*(g)=g_0+O(\md{\bs z})=
\sum_{j=1}^n\md{\d z_j}^2+O(\md{\bs z})$.
\end{itemize}
\label{hs3prop1}
\end{prop}

\begin{proof} Let $\ep'>0$, and for each $(p,\up)\in U$ define
$\Up'_{p,\up}:B_{\ep'}\ra M$ by $\Up'_{p,\up}=\exp_p\ci\up
\vert_{B_{\ep'}}$, where $\up\vert_{B_{\ep'}}:B_{\ep'}\ra T_pM$ and
$\exp_p:T_pM\ra M$ is geodesic normal coordinates at $p$ using $g$.
Then $\Up'_{p,\up}$ is a diffeomorphism near $0\in B_{\ep'}$ and
$p\in M$, so as $M,U$ are compact, if we take $\ep'>0$ small enough
$\Up'_{p,\up}$ is an embedding for all $(p,\up)\in U$. It is
immediate that $\Up_{p,\up}'$ is smooth in $p,\up$, and
$\Up'_{p,\up}(0)=p$, $\d\Up'_{p,\up}\vert_0=\up$, and
$\Up'_{p,\up\ci\ga}\equiv\Up'_{p,\up}\ci\ga$ for all $\ga\in\U(n)$.
Also $(\Up_{p,\up}')^*(\om)\vert_0=\om_0$ and
$(\Up_{p,\up}')^*(g)\vert_0=g_0$, so by Taylor's Theorem we have
$(\Up_{p,\up}')^*(\om)=\om_0+O(\md{\bs z})$ and
$(\Up_{p,\up}')^*(g)=g_0+O(\md{\bs z})$. Thus the $\ep'$ and
$\Up'_{p,\up}$ satisfy all the proposition except (iii), which is
replaced by~$(\Up_{p,\up}')^*(\om)=\om_0+O(\md{\bs z})$.

We shall use Moser's method for proving Darboux' Theorem in
\cite{Mose} to modify the $\Up'_{p,\up}$ to $\Up_{p,\up}$ with
$\Up_{p,\up}^* (\om)=\om_0$, preserving the other properties
(i),(ii),(iv). Define closed 2-forms $\om^s_{p,\up}$ on $B_{\ep'}$
for $(p,\up)\in U$ and $s\in[0,1]$ by
$\om^s_{p,\up}=(1-s)\om_0+s(\Up_{p,\up}')^*(\om)$. As $\om_0\vert_0=
(\Up_{p,\up}')^*(\om)\vert_0$, we have $\om^s_{p,\up}\vert_0=\om_0$,
so $\om^s_{p,\up}$ is symplectic near 0 in $B_{\ep'}$. Since
$[0,1]\t U$ is compact, by making $\ep'>0$ smaller we can suppose
$\om^s_{p,\up}$ is symplectic on $B_{\ep'}$ for all $s\in[0,1]$ and
$(p,\up)\in U$.

As $\om_0,(\Up_{p,\up}')^*(\om)$ are closed, we can choose a family
of 1-forms $\ze_{p,\up}$ on $B_{\ep'}$, smooth in $p,\up$, such that
$\d\ze_{p,\up}=\om_0-(\Up_{p,\up}')^*(\om)$. Since $\om_0\vert_0=
(\Up_{p,\up}')^*(\om)\vert_0$, we can choose the $\ze_{p,\up}$ to
satisfy $\md{\ze_{p,\up}}=O(\ms{\bs z})$. We also choose the family
$\ze_{p,\up}$ to be $\U(n)$-equivariant, in the sense that
$\ze_{p,\up\ci\ga}=\ga\vert_{B_{\ep'}}^*(\ze_{p,\up})$ for all
$\ga\in\U(n)$, noting that $\ga:\C^n\ra\C^n$ restricts to
$\ga\vert_{B_{\ep'}}:B_{\ep'}\ra B_{\ep'}$. To do this, we first
choose $\ze_{p,\up}'$ not necessarily $\U(n)$-equivariant, and then
take $\ze_{p,\up}$ to be the average of
$(\ga^{-1})^*(\ze_{p,\up\ci\ga}')$ over~$\ga\in\U(n)$.

Now let $v^s_{p,\up}$ be the unique vector field on $B_{\ep'}$ with
$v^s_{p,\up}\cdot\om^s_{p,\up}=\ze_{p,\up}$, which is well-defined
as $\om^s_{p,\up}$ is symplectic, and varies smoothly in $s,p,\up$.
Then $\d(v^s_{p,\up}\cdot\om^s_{p,\up})=\om_0-(\Up_{p,\up}')^*
(\om)$. As $\md{\ze_{p,\up}}=O(\ms{\bs z})$ we have
$\md{v_{p,\up}^s}=O(\ms{\bs z})$. For $0<\ep\le\ep'$ we construct a
family of embeddings $\vp^s_{p,\up}:B_\ep\ra B_{\ep'}$ with
$\vp^0_{p,\up}=\id:B_\ep\ra B_\ep\subset B_{\ep'}$ by solving the
o.d.e.\ $\frac{\d}{\d s}\vp_{p,\up}^s=v_{p,\up}^s\ci\vp_{p,\up}^s$.
By compactness of $[0,1]\t U$, this is possible provided $\ep>0$ is
small enough. Then $(\vp_{p,\up}^s)^*(\om^s_{p,\up})=\om_0$ for all
$s$, so that $(\vp_{p,\up}^1)^*\bigl((\Up_{p,\up}')^*(\om)\bigr)
=\om_0$. Also, as $\md{v_{p,\up}^s}=O(\ms{\bs z})$ we see that
$\vp_{p,\up}^s(0)=0$ and $\d\vp_{p,\up}^s\vert_0=\id:\C^n\ra\C^n$
for all $s\in[0,1]$. And the $\U(n)$-equivariance of the
$\ze_{p,\up}$ implies $\U(n)$-equivariance of the $v_{p,\up}^s$,
which in turn implies that $\vp_{p,\up\ci\ga}^s=
\ga^{-1}\ci\vp_{p,\up}^s\ci\ga$ for all $s\in[0,1]$, $(p,\up)\in U$
and~$\ga\in\U(n)$.

Define $\Up_{p,\up}=\Up_{p,\up}'\ci\vp_{p,\up}^1$. Then
$\Up_{p,\up}$ depends smoothly on $p,\up$. Part (i) holds as
$\vp_{p,\up}^1(0)=0$, $\d\vp_{p,\up}^1\vert_0=\id$,
$\Up'_{p,\up}(0)=p$ and $\d\Up'_{p,\up}\vert_0=\up$. Part (ii) holds
by $\U(n)$-equivariance of the $\Up'_{p,\up}$, and
$\vp_{p,\up\ci\ga}^1=\ga^{-1}\ci\vp_{p,\up}^1\ci\ga$ for all
$(p,\up)\in U$ and $\ga\in\U(n)$. Part (iii) follows from
$(\vp_{p,\up}^1)^*\bigl((\Up_{p,\up}')^*(\om)\bigr)=\om_0$, and (iv)
from $\vp_{p,\up}^1(0)=0$, $\d\vp_{p,\up}^1\vert_0=\id$,
and~$(\Up'_{p,\up})^*(g)=g_0+O(\md{\bs z})$.
\end{proof}

\begin{rem} When $(M,J,g)$ is a K\"ahler manifold with K\"ahler
form $\om$, by applying the same argument to holomorphic normal
coordinates, we can obtain the better approximation
$\Up_{p,\up}^*(\om)=\om_0=\ts\frac{i}{2}
\sum_{j=1}^n\d z_j\w\d\bar{z}_j$ and
\begin{equation*}
\Up_{p,\up}^*(g)=\ts \sum_{j=1}^n\md{\d z_j}^2+\frac{1}{2}
\sum_{i,j,k,l}\Re\bigl(R_{i\bar{j}k\bar{l}}(p)z^i z^k \d\bar z^{j}\d\bar
z^{l}\bigr)+O\bigl(\md{\bs z}^3\bigr).
\end{equation*}
\end{rem}

\subsection{Dilations, and uniform estimates of $t^{-2}(\Up_{p,\up}\ci
t)^*(g)$}
\label{hs32}

We continue to use the notation of \S\ref{hs31} and Proposition
\ref{hs3prop1}. Let $R>0$ be given. For $0<t\le R^{-1}\ep$, consider
the dilation map $t:B_R\ra B_\ep$ mapping $t:(z_1,\ldots,z_n)\mapsto
(tz_1,\ldots,tz_n)$. Then $\Up_{p,\up}\ci t$ is an embedding $B_R\ra
M$, so we can consider the pullbacks $(\Up_{p,\up}\ci t)^*(\om)$,
$(\Up_{p,\up}\ci t)^*(g)$. Since $t^*(\om_0)=t^2\om_0$ and
$t^*(g_0)=t^2g_0$, Proposition \ref{hs3prop1}(iii),(iv) give
\begin{equation}
(\Up_{p,\up}\ci t)^*(\om)=t^2\om_0 \quad\text{and}\quad
(\Up_{p,\up}\ci t)^*(g)=t^2g_0+O(t^3\md{\bs z}),
\label{hs3eq1}
\end{equation}
where the power $t^3=t^2\cdot t$ in $O(t^3\md{\bs z})$ combines the
fact that $t^*(\d z_i\d\bar z_j)=t^2\d z_i\d\bar z_j$, that is,
tensors of type (0,2) scale by $t^2$ under $t^*$, and regarding
$\md{\bs z}$ as a function on $\C^n$ we have $t^*(\md{\bs
z})=t\md{\bs z}$. Multiplying \eq{hs3eq1} by $t^{-2}$ gives
\begin{equation}
t^{-2}(\Up_{p,\up}\ci t)^*(\om)=\om_0 \quad\text{and}\quad
t^{-2}(\Up_{p,\up}\ci t)^*(g)=g_0+O(t\md{\bs z})\quad\text{on
$B_R$.}
\label{hs3eq2}
\end{equation}

Define a Riemannian metric $g^t_{p,\up}$ on $B_R$ by $g^t_{p,\up}=
t^{-2}(\Up_{p,\up}\ci t)^*(g)$. This depends smoothly on
$t\in(0,R^{-1}\ep]$ and $(p,\up)\in U$, and satisfies
$g^t_{p,\up}=g_0+O(t\md{\bs z})$ by \eq{hs3eq2}. Since $g$ is
compatible with $\om$, $t^{-2}(\Up_{p,\up}\ci t)^*(g)$ is compatible
with $t^{-2}(\Up_{p,\up}\ci t)^*(\om)$, and thus $g^t_{p,\up}$ is
compatible with the fixed symplectic form $\om_0$ on $B_R$ for all
$t,p,\up$. We prove {\it uniform estimates\/} on these
metrics~$g^t_{p,\up}$.

\begin{prop} There exist positive constants
$C_0,C_1,C_2,\ldots$ such that for all\/ $t\in(0,\ha R^{-1}\ep]$
and\/ $(p,\up)\in U,$ the metric $g^t_{p,\up}=t^{-2}(\Up_{p,\up}\ci
t)^*(g)$ on\/ $B_R$ satisfies the estimates
\begin{equation}
\nm{g^t_{p,\up}-g_0}_{C^0}\le C_0t\quad\text{and}\quad
\nm{\pd^kg^t_{p,\up}}_{C^0}\le C_kt^k \quad\text{for
$k=1,2,\ldots,$}
\label{hs3eq3}
\end{equation}
where norms are taken w.r.t.\ $g_0,$ and\/ $\pd$ is the Levi-Civita
connection of\/~$g_0$.
\label{hs3prop2}
\end{prop}

\begin{proof} We will prove \eq{hs3eq3} by first estimating the
metrics $\Up_{p,\up}^*(g)$ for all $(p,\up)\in U$. As $B_\ep$ is
noncompact, we cannot do this on the whole of $B_\ep$, since
$\Up_{p,\up}^*(g)$ might have bad behaviour approaching the boundary
of $B_\ep$, forcing norms to be unbounded. Instead, write
$\ovB_{\ep/2}$ for the closure of $B_{\ep/2}$ in $B_\ep$. Then
$\ovB_{\ep/2}$ is compact. For each fixed $(p,\up)\in U$,
Proposition \ref{hs3prop1}(iv) implies that there exists $C>0$ with
$\bmd{\Up_{p,\up}^*(g)-g_0}\le C\md{z}$ on $\ovB_{\ep/2}$, where
$\md{\,.\,}$ is taken using $g_0$. Since $U$ is compact, we can
choose $C>0$ such that this holds for all $(p,\up)\in U$. Now let
$t\in(0,\ha R^{-1}\ep]$. Then $t(B_R)\subseteq
B_{\ep/2}\subset\ovB_{\ep/2}$. By the proof of the second equation
of \eq{hs3eq2}, $\bmd{\Up_{p,\up}^*(g)-g_0}\le C\md{z}$ on
$\ovB_{\ep/2}$ implies that $\md{g^t_{p,\up}-g_0}\le Ct\md{z}\le
CRt$ on $B_R$, as $\md{z}\le R$ on $B_R$. Setting $C_0=CR$ proves
the first equation of~\eq{hs3eq3}.

Fix $k=1,2,\ldots$, and consider $\bmd{\pd^k\Up_{p,\up}^*(g)(\bs
z)}$. This is a continuous function of $(p,\up)\in U$ and $\bs z\in
\ovB_{\ep/2}$, so as $U,\ovB_{\ep/2}$ are compact there exists
$C_k>0$ with $\bmd{\pd^k\Up_{p,\up}^*(g)(\bs z)}\le C_k$ for all
$(p,\up)\in U$ and $\bs z\in\ovB_{\ep/2}$. Now let $t\in(0,\ha
R^{-1}\ep]$ and $\bs z\in B_R$. Then an easy scaling argument shows
that $\bmd{\pd^k\bigl(t^{-2}(\Up_{p,\up}\ci t)^*(g)\bigr)(\bs
z)}=t^k\bmd{\pd^k\Up_{p,\up}^*(g)(t\bs z)}$. As $t\bs z\in
B_{\ep/2}\subset\ovB_{\ep/2}$ we have
$\bmd{\pd^k\Up_{p,\up}^*(g)(t\bs z)}\le C_k$, and the second
equation of \eq{hs3eq3} follows.
\end{proof}

The proposition implies that by taking $t$ sufficiently small, we
can make $g^t_{p,\up}$ arbitrarily close to $g_0$ on $B_R$ uniformly
for all $(p,\up)\in U$, in the $C^k$ norm for any $k\ge 0$, and
hence also in the H\"older $C^{k,\ga}$ norm for any $k\ge 0$
and~$\ga\in(0,1)$.

\section{Setting up the problem}
\label{hs4}

In \S\ref{hs4}--\S\ref{hs6} we will prove Theorem A. This section
will set up a lot of notation, and formulate a family of
fourth-order nonlinear elliptic partial differential operators
$P^t_{p,\up}:C^\iy(L)\ra C^\iy(L)$ depending on $(p,\up)\in U$ and
small $t>0$, such that $C^1$-small $f\in C^\iy(L)$ correspond to
Lagrangians $L_{p,\up}^{t,f}$ in $M$, and $L_{p,\up}^{t,f}$ is
Hamiltonian stationary when $P^t_{p,\up}(f)=0$.

Section \ref{hs5} will show that for sufficiently small, fixed $t>0$
and all $(p,\up)\in U$ we can find a family of functions
$f^t_{p,\up}\in C^\iy(L)$ with $P^t_{p,\up}(f^t_{p,\up})\in\Ker{\cal
L}$ and $f^t_{p,\up}\perp\Ker{\cal L}$, which are unique for
$\nm{f^t_{p,\up}}_{C^{4,\ga}}$ (the H\"older $C^{4,\ga}$ norm)
small, and depend smoothly on $(p,\up)\in U$. Finally, \S\ref{hs6}
shows that the smooth map $H^t:U\ra\Ker{\cal L}$ defined by
$H^t:(p,\up)\mapsto P^t_{p,\up}(f^t_{p,\up})$ can be interpreted in
terms of the exact 1-form $\d K^t$ on $U$, where $K^t:U\ra\R$ is
defined by $K^t(p,\up)=t^{-n}\Vol_g\bigl(L_{p,\up}^t\bigr)$, with
$L_{p,\up}^t=L_{p,\up}^{t,\smash{f^{\smash{t}}_{\smash{p,\up}}}}$.
Thus, if $(p,\up)$ is a critical point of $K^t$ then
$L'=L_{p,\up}^t$ is Hamiltonian stationary, as we want.

Let $L$ be a nonempty, compact, rigid, Hamiltonian stationary
Lagrangian in $\C^n$. Then $L$ is connected as in \S\ref{hs23}. Let
$G$ be the subgroup of $\U(n)\lt\C^n$ preserving $L$. Then $G$ is
compact, as $L$ is compact, so it is a Lie subgroup of
$\U(n)\lt\C^n$, which acts on $L$. As $G$ is compact it must fix a
point in $\C^n$, the centre of gravity of $L$. Translating $L$ in
$\C^n$ if necessary, we suppose $G$ fixes 0 in $\C^n$, so
that~$G\subset\U(n)$.

By the Lagrangian Neighbourhood Theorem, Theorem \ref{hs2thm2}, we
can choose an open tubular neighbourhood $T$ of the zero section $L$
in $T^*L$, and an embedding $\Phi:T\ra\C^n$ with
$\Phi\vert_L=\id:L\ra L$ and $\Phi^*(\om_0)=\hat\om$, where
$\hat\om$ is the canonical symplectic structure on $T^*L$. Making
$T$ smaller if necessary, we suppose $T$ is of the form
\begin{equation}
T=\bigl\{(p,\al):\text{$p\in L$, $\al\in T_p^*L$,
$\md{\al}<\de$}\bigr\}
\label{hs4eq1}
\end{equation}
for some small $\de>0$, where $\md{\al}$ is computed using the
metric $g_0\vert_L$. The action of $G$ on $L$ induces an action of
$G$ on $T^*L$, and as $g_0\vert_L$ is $G$-invariant, $T$ is
$G$-invariant. We suppose $\Phi$ is chosen to be equivariant under
the actions of $G$ on $T$ and $\C^n$; this can be done following the
proof of the dilation-equivariant Lagrangian Neighbourhood Theorem
in~\cite[Th.~4.3]{Joyc}.

Fix $R>0$ such that $L\subset\Phi(T)\subset B_R\subset\C^n$. Now let
$f\in C^\iy(L)$ with $\nm{\d f}_{C^0}<\de$. Define the graph
$\Ga_{\d f}$ of $\d f$ in $T^*L$ to be $\Ga_{\d f}=\bigl\{(q,\d
f\vert_q):q\in L\bigr\}$. Then $\Ga_{\d f}$ is an embedded
Lagrangian submanifold of $T^*L$ diffeomorphic to $L$. As $\nm{\d
f}_{C^0}<\de$ we have $\md{\d f\vert_q}<\de$ for all $q\in L$, and
$\Ga_{\d f}\subset T$. Hence $\Phi(\Ga_{\d f})$ is a submanifold of
$\Phi(T)\subset B_R\subset\C^n$ diffeomorphic to $L$. Since $\Ga_{\d
f}$ is Lagrangian in $(T,\hat\om)$ and $\Phi^*(\om_0)=\hat\om$, we
see that $\Phi(\Ga_{\d f})$ is Lagrangian in $(B_R,\om_0)$.

Suppose $(M,\om)$ is a compact symplectic $2n$-manifold, and $g$ a
Riemannian metric on $M$ compatible with $\om$. Define $U$ as in
Definition \ref{hs3def1} and $\ep>0$, $\Up_{p,\up}:B_\ep\ra M$ for
$(p,\up)\in U$ as in Proposition \ref{hs3prop1}. For $0<t\le
R^{-1}\ep$ and $(p,\up)\in U$, define metrics $g^t_{p,\up}$ on $B_R$
as in~\S\ref{hs32}.

Let $0<t\le\ha R^{-1}\ep$. For each $f\in C^\iy(L)$ with $\nm{\d
f}_{C^0}<\de$, define $L_{p,\up}^{t,f}=\Up_{p,\up}\ci t\ci \Phi
(\Ga_{\d f})$. Then $L_{p,\up}^{t,f}$ is an embedded submanifold of
$M$ diffeomorphic to $L$, and as $\Phi(\Ga_{\d f})$ is Lagrangian in
$(B_R,\om_0)$ and $(\Up_{p,\up}\ci t)^*(\om)=t^2\om_0$ by
\eq{hs3eq1}, we see that $L_{p,\up}^{t,f}$ is Lagrangian in
$(M,\om)$. Define a functional
\begin{equation}
F^t_{p,\up}:\bigl\{f\in C^\iy(L):\nm{\d f}_{C^0}<\de\bigr\}\ra\R
\quad\text{by}\quad F^t_{p,\up}(f)=t^{-n}\Vol_g
\bigl(L_{p,\up}^{t,f}\bigr).
\label{hs4eq2}
\end{equation}

We have
\begin{equation}
F^t_{p,\up}(f)\!=\!t^{-n}\Vol_g\bigl(L_{p,\up}^{t,f})\!=\!t^{-n}\Vol_{
(\Up_{p,\up}\ci t)^*(g)}\bigl(\Phi(\Ga_{\d f})\bigr)\!=\!
\Vol_{g^t_{p,\up}}\bigl(\Phi(\Ga_{\d f})\bigr),\!\!
\label{hs4eq3}
\end{equation}
since $g^t_{p,\up}=t^{-2}(\Up_{p,\up}\ci t)^*(g)$ and $\dim L=n$.
Proposition \ref{hs3prop2} shows that $g^t_{p,\up}$ approximates
$g_0$ when $t$ is small. Thus \eq{hs4eq3} implies that
\begin{equation}
F^t_{p,\up}(f)\approx \Vol_{g_0}\bigl(\Phi(\Ga_{\d f})\bigr)
\label{hs4eq4}
\end{equation}
for small $t>0$, where the right hand side is independent
of~$t,p,\up$.

Observe that for fixed $t,p,\up$ and varying $f$, the Lagrangians
$L_{p,\up}^{t,f}$ in $(M,\om)$ are all Hamiltonian equivalent, and
furthermore, for any fixed $f$ the $\smash{L_{p,\up}^{t,f'}}$ for
$f'$ close to $f$ in $C^\iy(L)$ realize all Hamiltonian equivalent
Lagrangians close to $L_{p,\up}^{t,f}$. This parametrization of
Lagrangians $L_{p,\up}^{t,f}$ by functions $f$ is degenerate, since
$L_{p,\up}^{t,f}=L_{p,\up}^{t,f+c}$ for $c\in\R$. We can make it a
local isomorphism by restricting to $f\in C^\iy(L)$ with
$\int_Lf\,\d V_{g_0\vert_L}=0$. Anyway, $L_{p,\up}^{t,f}$ is
Hamiltonian stationary in $M$ if and only if $f$ is a critical point
of the functional $F^t_{p,\up}$, and $L_{p,\up}^{t,f}$ is
Hamiltonian stable if in addition the second variation of
$F^t_{p,\up}$ at $f$ is nonnegative.

From \eq{hs4eq3} we see that
\begin{equation}
F^t_{p,\up}(f)=\int_{\bs \Phi(\Ga_{\d f})}\d V_{g^t_{p,\up}
\vert_{\Phi(\Ga_{\d f})}}=\int_{ L}(\Phi_f)^*\bigl(\d
V_{g^t_{p,\up}\vert_{\Phi(\Ga_{\d f})}}\bigr).
\label{hs4eq5}
\end{equation}
Here in the first step $g^t_{p,\up}\vert_{\Phi(\Ga_{\d f})}$ is the
restriction of the metric $g^t_{p,\up}$ on $B_R$ to the submanifold
$\Phi(\Ga_{\d f})$, and $\d V_{g^t_{p,\up}\vert_{\Phi(\Ga_{\d f})}}$
is the induced volume form on $\Phi(\Ga_{\d f})$. Now $\Phi(\Ga_{\d
f})$ is the diffeomorphic image of $L$ under the map
$\Phi_f:q\mapsto\Phi(q,\d f\vert_q)$. In the second step of
\eq{hs4eq5}, we  pull back the volume form by $\Phi _f$ and do the
integration on $L$.

At each $q\in L$, the integrand $(\Phi_f)^*\bigl(\d
V_{g^t_{p,\up}\vert_{\Phi(\Ga_{\d f})}}\bigr)$ in \eq{hs4eq5} is a
positive multiple of the volume form $\d V_{g_0\vert_L}$ of the
metric $g_0\vert_L$ on $L$ at $q$, and this multiple depends only on
$t,p,\up,q,\d f\vert_q$ and $\nabla\d f\vert_q$, where $\nabla$ is
the Levi-Civita connection of $g_0\vert_L$. This is because the
tangent space to $\Phi(\Ga_{\d f})$ at $\Phi(q,\d f\vert_q)$ depends
only on $q,\d f\vert_q,\nabla\d f\vert_q$, and the volume form
depends on this tangent space and on $g^t_{p,\up}$, which depends on
$t,p,\up$. Thus we may write
\begin{equation}
F^t_{p,\up}(f)=\int_{L}G^t_{p,\up}\bigl(q,\d f\vert_q,\nabla\d
f\vert_q\bigr)\,\d V_{g_0\vert_L}.
\label{hs4eq6}
\end{equation}
Here $G^t_{p,\up}$ maps
\begin{equation}
G^t_{p,\up}:\bigl\{(q,\al,\be):\text{$q\in L$, $\al\in T^*_qL$,
$\md{\al}<\de$, $\be\in S^2T^*_qL$}\bigr\}\longra\R,
\label{hs4eq7}
\end{equation}
where the condition $\md{\al}<\de$ is because we restrict to $f$
with $\nm{\d f}_{C^0}<\de$ so that $\Ga_{\d f}\subset T$, and we
take $\be\in S^2T^*_qL\subset\ot^2T^*_qL$ because the second
derivative $\nabla\d f$ is a symmetric tensor. Clearly,
$G^t_{p,\up}(q,\al,\be)$ is a smooth, nonlinear function of its
arguments $q,\al,\be$, and also depends smoothly on $t\in(0,\ha
R^{-1}\ep]$ and~$(p,\up)\in U$.

From above, $L_{p,\up}^{t,f}$ is Hamiltonian stationary in $M$ if
and only if $f$ is a critical point of the functional $F^t_{p,\up}$.
Applying the Euler--Lagrange method to \eq{hs4eq6}, we see that
$L_{p,\up}^{t,f}$ is Hamiltonian stationary in $M$ if and only if
$P^t_{p,\up}(f)=0$, where
\begin{equation}
P^t_{p,\up}:\bigl\{f\in C^\iy(L):\nm{\d f}_{C^0}<\de\bigr\}\longra
C^\iy(L)
\label{hs4eq8}
\end{equation}
is defined by
\begin{equation}
\begin{split}
P^t_{p,\up}(f)(q)&=\d^*\bigl((\ts\frac{\pd}{\pd\al}G^t_{p,\up})(q,\d
f\vert_q,\nabla\d f\vert_q)\bigr)\\
&\qquad +\d^*\bigl(\nabla^*\bigl((\ts\frac{\pd}{\pd\be}
G^t_{p,\up})(q,\d f\vert_q,\nabla\d f\vert_q)\bigr)\bigr).
\end{split}
\label{hs4eq9}
\end{equation}

Here we consider $G^t_{p,\up}(q,\al,\be)$ as a function of $q\in L$
and $\al\in T^*_qL$ and $\be\in S^2T^*_qL$, so that $\frac{\pd}{\pd
\al}G^t_{p,\up}$, $\frac{\pd}{\pd\be}G^t_{p,\up}$ are the partial
derivatives of $G^t_{p,\up}$ in the $\al,\be$ directions. Having
defined these partial derivatives, we then set $\al=\d f\vert_q$,
$\be=\nabla\d f\vert_q$ and regard $(\frac{\pd}{\pd\al}G^t_{p,\up})
(q,\d f\vert_q,\nabla\d f\vert_q)$, $(\frac{\pd}{\pd\be}
G^t_{p,\up})(q,\d f\vert_q,\nabla\d f\vert_q)$ as tensor fields on
$L$ depending on $q\in L$, and we apply $\d^*$ and $\d^*\ci\nabla^*$
to them to get functions on $L$. The first term on the right hand
side of \eq{hs4eq9} involves three derivatives of $f$, and the
second term four derivatives. Hence $P^t_{p,\up}$ is a fourth-order
nonlinear partial operator, which is in fact quasilinear and
elliptic at $f$ for all $t\in (0,\ha R^{-1}\ep]$ and $f\in C^\iy(L)$
with~$\nm{\d f}_{C^0}<\de$.

As for \eq{hs4eq2}, define
\begin{equation}
F_0:\bigl\{f\in C^\iy(L):\nm{\d f}_{C^0}<\de\bigr\}\ra\R
\quad\text{by}\quad F_0(f)=\Vol_{g_0}\bigl(\Phi(\Ga_{\d f})\bigr).
\label{hs4eq10}
\end{equation}
The proof of \eq{hs4eq6}--\eq{hs4eq7} shows that we may write
\begin{equation}
F_0(f)=\int_{q\in L}G_0\bigl(q,\d f\vert_q,\nabla\d
f\vert_q\bigr)\,\d V_{g_0\vert_L},
\label{hs4eq11}
\end{equation}
where $G_0$ is a smooth nonlinear map
\begin{equation}
G_0:\bigl\{(q,\al,\be):\text{$q\in L$, $\al\in T^*_qL$,
$\md{\al}<\de$, $\be\in S^2T^*_qL$}\bigr\}\longra\R.
\label{hs4eq12}
\end{equation}
As for \eq{hs4eq8}--\eq{hs4eq9}, define
\begin{equation}
\begin{split}
&P_0:\bigl\{f\in C^\iy(L):\nm{\d f}_{C^0}<\de\bigr\}\longra
C^\iy(L)\quad\text{by}\\
&P_0(f)(q)=\d^*\bigl((\ts\frac{\pd}{\pd\al}G_0)(q,\d
f\vert_q,\nabla\d f\vert_q)\bigr)\\
&\qquad\qquad\qquad +\d^*\bigl(\nabla^*\bigl((\ts\frac{\pd}{\pd\be}
G_0)(q,\d f\vert_q,\nabla\d f\vert_q)\bigr)\bigr).
\end{split}
\label{hs4eq14}
\end{equation}
Then \eq{hs4eq4} says that $F^t_{p,\up}(f)\approx F_0(f)$ for small
$t$. A similar proof shows that $G^t_{p,\up}(q,\al,\be)\approx
G_0(q,\al,\be)$ and $P^t_{p,\up}(f)\approx P_0(f)$ for small $t$. In
fact, the difference between $P^t_{p,\up}$ and $P_0$ depends only on
the difference between $g^t_{p,\up}$ and $g_0$, and on finitely many
derivatives of this. Therefore, from Proposition \ref{hs3prop2} we
deduce:

\begin{prop} Let any\/ $k\ge 0,$ $\ga\in(0,1),$ and small\/ $C>0$
and $\ze>0$ be given. Then if\/ $t>0$ is sufficiently small, for
all\/ $f\in C^{k+4,\ga}(L)$ with\/ $\nm{\d f}_{C^0}\le\ha\de$ and
$\nm{\nabla \d f}_{C^0}\le C,$ and all $(p,\up)\in U$ we have
\begin{equation}
\bnm{P^t_{p,\up}(f)-P_0(f)}{}_{C^{k,\ga}}\le\ze
\quad\text{and\/}\quad \bnm{{\cal L}^t_{p,\up}(f)-{\cal
L}(f)}{}_{C^{k,\ga}}\le \ze \nm{f}_{C^{k+4,\ga}},
\label{hs4eq15}
\end{equation}
where ${\cal L}^t_{p,\up}$ denotes the linearization of
$P^t_{p,\up}$ at $0$. That is, by taking $t$ small we can suppose
$P^t_{p,\up}$ and its linearization at $0$ are arbitrarily close to
$P_0$ and its linearization at $0$ as operators $C^{k+4,\ga}(L)\ra
C^{k,\ga}(L)$ (on their respective domains) uniformly
in\/~$(p,\up)\in U$.
\label{hs4prop}
\end{prop}

We impose the conditions $\nm{\d f}_{C^0}\le\ha\de$ and
$\nm{\nabla\d f}_{C^0}\le C$ so that we restrict to a compact subset
of the domains of $G^t_{p,\up},G_0$ in \eq{hs4eq7}, \eq{hs4eq12},
and then we use Proposition \ref{hs3prop2} to bound the difference
between $G^t_{p,\up}$ and $G_0$ in $C^{k+2,\ga}$ on this compact
subset.

Finally, we note that $P_0(f)=0$ is the Euler--Lagrange equation for
stationary points of the functional $F_0(f)$. Thus, $P_0(f)=0$ if
and only if $\Phi(\Ga_{\d f})$ is Hamiltonian stationary in $\C^n$.
But when $f=0$, $\Phi(\Ga_{0})=L$ which is Hamiltonian stationary in
$\C^n$, by assumption. Hence $P_0(0)=0$. Also, as in \S\ref{hs22},
the linearization of $P_0$ at $f=0$ is $\cal L$ in~\eq{hs2eq7}.

\section{Solving the family of p.d.e.s mod $\Ker{\cal L}$}
\label{hs5}

Our ultimate goal is to show that there exists $f^t_{p,\up}\in
C^\iy(L)$ with $P^t_{p,\up}(f^t_{p,\up})=0$ for all small $t>0$ and
some $(p,\up)\in U$ depending on $t$. As an intermediate step we
will show that we can solve the equation
$P^t_{p,\up}(f^t_{p,\up})=0$ mod $\Ker{\cal L}$ for all small $t>0$
and all $(p,\up)\in U$, and the solution is unique provided
$f^t_{p,\up}$ is orthogonal to $\Ker{\cal L}$ and small
in~$C^{4,\ga}$.

\begin{thm} In the situation of\/ \S{\rm\ref{hs4},} suppose $0<t\le\ha
R^{-1}\ep$ is sufficiently small and fixed. Then for all\/
$(p,\up)\in U,$ there exists $f^t_{p,\up}\in C^\iy(L)$ satisfying
\begin{equation}
P^t_{p,\up}(f^t_{p,\up})\in\Ker{\cal L} \quad\text{and\/}\quad
f^t_{p,\up}\perp \Ker{\cal L},
\label{hs5eq1}
\end{equation}
where $f^t_{p,\up}\perp \Ker{\cal L}$ means $f^t_{p,\up}$ is
$L^2$-orthogonal to $\Ker{\cal L}$. Furthermore $f^t_{p,\up}$ is the
unique solution of \eq{hs5eq1} with $\nm{f^t_{p,\up}}_{C^{4,\ga}}$
small, and\/ $f^t_{p,\up}$ depends smoothly on~$(p,\up)\in U$.
\label{hs5thm}
\end{thm}

\begin{proof} Let $X_1$ denote the Banach space of functions $f\in
C^{4,\ga}(L)$ which are orthogonal to $\Ker{\cal L}$, and let $X_2$
denote the Banach subspace of $C^{0,\ga}(L)$ consisting of functions
which are orthogonal to $\Ker{\cal L}$. The starting point of the
proof is the observation that the operator $\cal L$ is a bounded
linear isomorphism from $X_1$ to $X_2$ with bounded inverse. This
follows directly from the self-adjointness and ellipticity of $\cal
L$. We let $\Pi$ denote orthogonal projection from $L^2(L)$ to the
orthogonal complement of $\Ker{\cal L}$. We will show that for $t$
sufficiently small and for all $(p,\up)\in U$ there is a unique
small solution $f^t_{p,\up}\in X_1$ of $\Pi\ci P^t_{p,\up}
(f^t_{p,\up})=0$, which depends smoothly on~$(t,p,\up)$.

The first step is to show that there are $t_0,r_0>0$ sufficiently
small so that for all $t\in (0,t_0)$ and for all $(p,\up)\in U$
there is a unique solution $f^t_{p,\up}\in B_{r_0}(0)$ of $\Pi\ci
P^t_{p,\up}(f^t_{p,\up})=0$. To accomplish this we consider the
smooth map $F=\Pi\ci P^t_{p,\up}$ from a neighbourhood of the origin
in $X_1$ to $X_2$. The derivative of $F$ at $0$ is $\Pi\ci {\cal
L}^t_{p,\up}$, and by Proposition \ref{hs4prop} this is close in the
operator norm to $\cal L$ for $t$ sufficiently small, and therefore
is a linear isomorphism with bounded inverse (note that $\Pi\ci{\cal
L}={\cal L}$). The standard contraction mapping argument for proving
the Inverse Function Theorem implies that the map $F$ is a
diffeomorphism from a ball of radius $r_0$ centered at $0$ in $X_1$
onto a domain containing the ball of radius $\la r_0$ about $F(0)$,
where $\la=(2\nm{F'(0)^{-1}})^{-1}$, and the radius $r_0$ can be
estimated below in terms of the norm of $F'(0)^{-1}$ and the modulus
of continuity of $F'$. This result may be found in \cite[\S
VI.1]{Lang}. Thus by the first inequality of \eq{hs4eq15} with $f=0$
(note that $P_0(0)=0$), there is a $t_0$ sufficiently small so that
$0$ lies in the ball of radius $\la r_0$ centered at $F(0)$ for
$t\in (0,t_0)$. This gives us a unique small solution $f^t_{p,\up}$
of $\Pi\ci P^t_{p,\up}(f^t_{p,\up})=0$, as claimed.

The next step is to show that the solutions $f^t_{p,\up}$ depend
smoothly on the parameters $(t,p,\up)$. We will do this by using the
Implicit Function Theorem (see \cite[\S VI.2]{Lang}). Precisely, we
consider the smooth map $G$ from $(0,t_0)\t U\t B_{r_0}(0)$ to $X_2$
given by $G(t,p,\up,f)=\Pi\ci P^t_{p,\up}(f)$. We need to analyze
the set $G(t,p,\up,f)=0$, and we observe that the derivative in the
$f$ variable is a linear isomorphism with bounded inverse. Thus the
Implicit Function Theorem implies that this zero set is a smooth
graph $(t,p,\up)\to f^t_{p,\up}$ in a neighbourhood of any chosen
point $(s,q,\up_1,f^s_{q,\up_1})$ of the zero set.

Finally it follows from elliptic regularity theory that the
solutions $f^t_{p,\up}$ are actually in $C^\infty(L)$. This is
because they are $C^{4,\ga}$ solutions of the quasilinear elliptic
equation $P^t_{p,\up}(f)=k$, where $k\in\Ker{\cal L}$ is a smooth
function, and we may improve the regularity by using linear elliptic
estimates in a standard way. This completes the proof.
\end{proof}

Note that nothing in \S\ref{hs3}--\S\ref{hs5} uses the assumption
that $L$ is Hamiltonian rigid, only that it is Hamiltonian
stationary. So Theorem \ref{hs5thm} holds for general Hamiltonian
stationary $L$ in $\C^n$. We will use Hamiltonian rigidity
in~\S\ref{hs6}.

\section{Completing the proof of Theorem A}
\label{hs6}

We work in the situation of \S\ref{hs3}--\S\ref{hs5}. Let $t>0$ be
sufficiently small and fixed, and $f^t_{p,\up}\in C^\iy(L)$ for
$(p,\up)\in U$ be as in Theorem \ref{hs5thm}. Define
$L_{p,\up}^t=L_{p,\up}^{t,\smash{f^{\smash{t}}_{\smash{p,\up}}}}$
for $(p,\up)\in U$. Define a smooth function $K^t:U\ra\R$ by
$K^t(p,\up)=t^{-n}\Vol_g\bigl( L_{p,\up}^t\bigr)$. Define a smooth
map $H^t:U\ra\Ker{\cal L}$ by $H^t:(p,\up)\mapsto P^t_{p,\up}
(f^t_{p,\up})$. We will show that we can express $H^t$ in terms of
the exact 1-form $\d K^t$ on~$U$.

Recall that $G$ is the Lie subgroup of $\U(n)$ preserving $L$, and
that $U$ is a principal $\U(n)$-bundle over $M$, so that $\U(n)$ and
hence $G$ act on $U$. Also the operator $\cal L$ of \eq{hs2eq7} is
equivariant under the action of $G$ on $L$, since $G$ preserves all
the geometric data used to define $\cal L$, so the action of $G$ on
$C^\iy(L)$ restricts to an action of $G$ on~$\Ker{\cal L}$.

\begin{lem} For all $(p,\up)\in U$ and $\ga\in G$ we have
$f^t_{p,\up\ci\ga}\equiv f^t_{p,\up}\ci\ga$ as maps $L\ra\R$. The
function $K^t:U\ra\R$ is $G$-invariant, and the function
$H^t:U\ra\Ker{\cal L}$ is $G$-equivariant, under the natural actions
of\/ $G$ on $U$ and\/~$\Ker{\cal L}$.
\label{hs6lem}
\end{lem}

\begin{proof} Let $(p,\up)\in U$, $\ga\in G$ and $f\in C^\iy(L)$.
Then
\begin{align}
&F^t_{p,\up}(f\ci\ga^{-1})=t^{-n}\Vol_g\bigl(L_{p,\up}^{t,f\ci
\ga^{-1}})= t^{-n}\Vol_g \bigl(\Up_{p,\up}\ci t\ci\Phi (\Ga_{\d
(f\ci\ga^{-1})})\bigr)
\nonumber\\
&\;=t^{-n}\Vol_g \bigl(\Up_{p,\up}\ci t\ci\ga\ci\Phi (\Ga_{\d
f})\bigr)=t^{-n}\Vol_g \bigl(\Up_{p,\up}\ci\ga\ci t\ci\Phi
(\Ga_{\d f})\bigr)
\nonumber\\
&\;=t^{-n}\Vol_g \bigl(\Up_{p,\up\ci\ga}\ci t\ci\Phi (\Ga_{\d
f})\bigr)=t^{-n}\Vol_g\bigl(L_{p,\up\ci\ga}^{t,f})=
F^t_{p,\up\ci\ga}(f),
\label{hs6eq1}
\end{align}
using \eq{hs4eq2} in the first and seventh steps, the definition of
$L_{p, \up}^{t,f}$ in the second and sixth, $G$-equivariance of
$\Phi$ in the third, that $\ga$ and the dilation $t$ commute in the
fourth, and Proposition \ref{hs3prop1}(ii) in the fifth. Since
$P^t_{p,\up}$ is the Euler--Lagrange variation of $F^t_{p,\up}$, we
deduce that
\begin{equation}
P^t_{p,\up}(f\ci\ga^{-1})=\bigl(P^t_{p,\up\ci\ga}(f)\bigr)\ci\ga^{-1}.
\label{hs6eq2}
\end{equation}

Applying \eq{hs6eq2} to $f=f^t_{p,\up\ci\ga}$, since
$P^t_{p,\up\ci\ga}(f^t_{p,\up\ci\ga})\in\Ker{\cal L}$ which is
$G$-invariant, we see that $P^t_{p,\up}(f^t_{p,\up\ci\ga}
\ci\ga^{-1})\in\Ker{\cal L}$. Also $f^t_{p,\up\ci\ga}\perp\Ker{\cal
L}$, so $f^t_{p,\up\ci\ga}\ci\ga^{-1}\perp\Ker{\cal L}$, by
$G$-invariance of $\Ker{\cal L}$ and the $L^2$-inner product. Hence
both $f^t_{p,\up}$ and $f^t_{p,\up\ci\ga}\ci\ga^{-1}$ satisfy
\eq{hs5eq1}, so by uniqueness in Theorem \ref{hs5thm} we have
$f^t_{p,\up}=f^t_{p,\up\ci\ga}\ci\ga^{-1}$, which proves
$f^t_{p,\up\ci\ga}\equiv f^t_{p,\up}\ci\ga$ as we want.

Substituting $f=f^t_{p,\up\ci\ga}$ into \eq{hs6eq1} and using
$f^t_{p,\up}=f^t_{p,\up\ci\ga}\ci\ga^{-1}$ now gives
\begin{equation*}
K^t(p,\up)=t^{-n}\Vol_g\bigl(L_{p,\up}^{t, f^t_{p,\up}}\bigr)=
F^t_{p,\up}(f^t_{p,\up})=F^t_{p,\up\ci\ga}(f^t_{p,\up\ci\ga})=
K^t(p,\up\ci\ga),
\end{equation*}
so $K^t$ is $G$-invariant. Equivariance of $H^t$ follows from
$f^t_{p,\up\ci\ga}\equiv f^t_{p,\up}\ci\ga$ and
equation~\eq{hs6eq2}.
\end{proof}

Let $(p,\up)\in U$, and consider the tangent space $T_{(p,\up)}U$.
Now $U$ is a principal $\U(n)$-bundle over $M$, and the metric $g$
induces a natural connection on this principal bundle, so we have a
splitting $T_{(p,\up)}U=V_p\op H_p$, where $V_p,H_p$ are the
vertical and horizontal subspaces. The projection $\pi:U\ra M$
induces $\d\pi\vert_p:T_{(p,\up)}U\ra T_pM$ which has kernel $V_p$
and induces an isomorphism $H_p\ra T_pM$. Also $V_p$ is the tangent
space to the fibre of $\pi$ over $p$, which is a free orbit of
$\U(n)$. Thus the $\U(n)$-action induces an isomorphism $\u(n)\cong
V_p$. But $\up$ is an isomorphism $\C^n\ra T_pM$, so we have
isomorphisms $H_p\cong T_pM\cong\C^n$. Putting these together gives
a natural isomorphism $T_{(p,\up)}U\cong\u(n)\op\C^n$, where
$\u(n)\op\C^n$ is the Lie algebra of the symmetry group
$\U(n)\lt\C^n$ of~$\C^n$.

The 1-form $\d K^t$ on $U$ lies in $T_{(p,\up)}^*U\cong
(\u(n)\op\C^n)^*$ at $(p,\up)\in U$. As $K^t$ is $G$-invariant, $\d
K^t$ contracts to zero with the vector fields of the Lie algebra
$\g$ of $G$. Hence under the identification $T_{(p,\up)}^*U\cong
(\u(n)\op\C^n)^*$, $\d K^t\vert_{(p,\up)}$ lies in the annihilator
$\g^\ci$ of $\g$ in $(\u(n)\op\C^n)^*$. Thus, we can regard $\d K^t$
as a smooth function~$U\ra {\g}^\ci$.

Now consider the function $H^t:U\ra\Ker{\cal L}$ defined by
$H^t:(p,\up)\mapsto P^t_{p,\up}(f^t_{p,\up})$. From \eq{hs4eq9} we
see that $P^t_{p,\up}(f^t_{p,\up})$ is of the form $\d^*\eta$ for
some 1-form $\eta$ on $L$. Hence $\int_LP^t_{p,\up}(f^t_{p,\up})\,\d
V_{g_0\vert_L}=\an{1,\d^*\eta}_{L^2}=\an{\d 1,\eta}_{L^2}=0$.
Therefore $H^t$ maps $U$ to the subspace $\bigl\{f\in\Ker{\cal
L}:\int_Lf\,\d V_{g_0\vert_L}=0\bigr\}$ in $\Ker{\cal L}$. Since
$\Ker{\cal L}$ contains the constants, this subspace has codimension
1 in~$\Ker{\cal L}$.

We now for the first time use the assumption that $L$ is Hamiltonian
rigid. By definition, equality holds in \eq{hs2eq9}, so Lemma
\ref{hs2lem1} implies that $\dim\Ker{\cal L}=n^2+2n+1-\dim G$.
Therefore
\begin{equation}
\begin{split}
\ts\dim\bigl\{f\in\Ker{\cal L}:\int_Lf\,\d
V_{g_0\vert_L}=0\bigr\}&=n^2+2n
-\dim G\\
&=\dim(\u(n)\op\C^n)-\dim\g=\dim \g^\ci.
\end{split}
\label{hs6eq3}
\end{equation}

We will construct an isomorphism $\psi^t:\bigl\{f\in\Ker{\cal
L}:\int_Lf\,\d V_{g_0\vert_L}=0\bigr\}\ra\g^\ci$, which explains
\eq{hs6eq3}. Define a linear map
$\xi^t:\u(n)\op\C^n\ra\bigl\{f\in\Ker{\cal L}:\int_Lf\,\d
V_{g_0\vert_L}=0\bigr\}$ by $\xi^t:x\mapsto\mu_x\ci t\vert_L$, where
$t$ acts by dilations on $\C^n$, and $\mu_x:\C^n\ra\R$ is the unique
moment map for the vector field $v_x$ associated with $x\in
\u(n)\op\C^n$ with $\int_L(\mu_x\ci t)\,\d V_{g_0\vert_L}=0$. Since
$G\subset\U(n)$, it commutes with the dilation $t$, and as $G$ is
the subgroup of $\U(n)\lt\C^n$ fixing $L$, it is also the subgroup
of $\U(n)\lt\C^n$ fixing $tL$. Hence $\Ker\xi^t$ is the Lie algebra
$\g$ of $G$. As $L$ is Hamiltonian rigid, $\xi^t$ is surjective.
Using the $L^2$-inner product to identify $\bigl\{f\in\Ker{\cal
L}:\int_Lf\,\d V_{g_0\vert_L}=0\bigr\}$ with its dual, define
$\psi^t:\bigl\{f\in\Ker{\cal L}:\int_Lf\,\d
V_{g_0\vert_L}=0\bigr\}\ra (\u(n)\op\C^n)^*$ to be the dual map of
$\xi^t$. Then $\psi^t$ maps to $\g^\ci$, as $\Ker\xi^t=\g$, and is
injective, as $\xi^t$ is surjective. Equation \eq{hs6eq3} thus
implies that $\psi^t$ is an isomorphism.

\begin{prop} Regard\/ $\d K^t$ as a smooth function $U\ra {\mathfrak
g}^\ci,$ and\/ $H^t$ as a smooth function $U\ra \bigl\{f\in\Ker{\cal
L}:\int_Lf\,\d V_{g_0\vert_L}=0\bigr\},$ as above. Then for small\/
$t$ and all\/ $(p,\up)\in U$ there is an isomorphism\/
$\Psi^t_{p,\up}:\bigl\{f\in\Ker{\cal L}:\int_Lf\,\d
V_{g_0\vert_L}=0\bigr\}\ra \g^\ci$ such that\/ $\d
K^t\vert_{(p,\up)}=\Psi^t_{p,\up}\ci H^t(p,\up)$. Furthermore
$\Psi^t_{p,\up}$ approximates $\psi^t:\bigl\{f\in\Ker{\cal
L}:\int_Lf\,\d V_{g_0\vert_L}=0\bigr\}\ra {\mathfrak g}^\ci$ above,
and depends smoothly on~$p,\up$.
\label{hs6prop}
\end{prop}

\begin{proof} For all $(p,\up)\in U$, define $\io_{p,\up}^t:L\ra M$
by $\io_{p,\up}^t(q)=\Up_{p,\up}\ci t\ci\Phi(q,\d f^t_{p,\up}
\vert_q)$. Then $\io_{p,\up}^t$ is a Lagrangian embedding with
$\io_{p,\up}^t(L)=L_{p,\up}^t$. By construction, with $t$ fixed, the
family of Lagrangians $\io_{p,\up}^t(L)$ for $(p,\up)$ are all
Hamiltonian equivalent. Let $(p,\up)\in U$ and $x\in T_{(p,\up)}U$.
Consider the derivative of the family of maps $\io_{p',\up'}^t:L\ra
M$ for $(p',\up')\in U$ in direction $x$ at $(p',\up')=(p,\up)$ in
$U$. This gives $\pd_x\io_{p,\up}^t\in
C^\iy\bigl((\io_{p,\up}^t)^*(TM)\bigr)$, that is,
$\pd_x\io_{p,\up}^t$ is a section of the vector bundle
$(\io_{p,\up}^t)^*(TM)\ra L$.

Since the family $\io_{p,\up}^t:L\ra M$ for $(p,\up)\in U$ are
Hamiltonian equivalent Lagrangian embeddings, $\pd_x\io_{p,\up}^t$
is a Hamiltonian variation of $\io_{p,\up}^t(L)$. Hence
$(\pd_x\io_{p,\up}^t\cdot\om)\vert_{\io_{p,\up}^t(L)}$ is an exact
1-form on $\io_{p,\up}^t(L)$, and $(\io_{p,\up}^t)^*(\pd_x
\io_{p,\up}^t\cdot\om)$ is an exact 1-form on $L$. Since $L$ is
connected, as in \S\ref{hs23}, there is a unique function
$h_{p,\up}^t(x)\in C^\iy(L)$ with $\int_Lh_{p,\up}^t(x)\,\d
V_{g_0\vert_L}=0$ such that
$(\io_{p,\up}^t)^*(\pd_x\io_{p,\up}^t\cdot\om)=\d (h_{p,\up}^t(x))$.
This $h_{p,\up}^t(x)$ depends linearly on $x\in T_{(p,\up)}U$, so we
have defined a linear map~$h_{p,\up}^t: T_{(p,\up)}U\ra C^\iy(L)$.

By construction, if $\ga\in G$ then $\io_{p,\up\ci\ga}^t(L)=
\io_{p,\up}^t(L)$ as submanifolds of $M$, although the actual
parametrizations $\io_{p,\up\ci\ga}^t,\io_{p,\up}^t$ may differ. It
follows under the identification $T_{(p,\up)}U\cong\u(n)\op\C^n$, if
$x$ lies in the Lie subalgebra $\g$ of $\u(n)\op\C^n$, then
$h_{p,\up}^t(x)\equiv 0$, since $\pd_x\io_{p,\up}^t$ is an
infinitesimal reparametrization of a fixed Lagrangian
$\io_{p,\up}^t(L)$ in~$M$.

We have
\begin{equation}
\begin{split}
(\d K^t\vert_{(p,\up)})\cdot x&=\pd_x\bigl[ t^{-n}\Vol_g\bigl(
\Up_{p,\up}\ci t\ci\Phi (\Ga_{\d f^t_{p,\up}})\bigr)\bigr]\\
&=\ban{h_{p,\up}^t(x), P^t_{p,\up}(f^t_{p,\up})}_{L^2},
\end{split}
\label{hs6eq4}
\end{equation}
by definition of $K^t$, and using the fact that $P^t_{p,\up}(f)$ is
the Euler--Lagrange variation of $f\mapsto
t^{-n}\Vol_g\bigl(\Up_{p,\up}\ci t\ci\Phi(\Ga_{\d
f^t_{p,\up}})\bigr)$, so that for $h\in C^\iy(L)$ we have
\begin{equation*}
\ts\frac{\d}{\d s}\bigl[t^{-n}\Vol_g\bigl(\Up_{p,\up}\ci
t\ci\Phi(\Ga_{\d (f+sh)})\bigr)\bigr]\big\vert_{s=0}
=\ban{h,P^t_{p,\up}(f)}{}_{L^2}.
\end{equation*}

Define a linear map $\Psi^t_{p,\up}:\bigl\{f\in\Ker{\cal
L}:\int_Lf\,\d V_{g_0\vert_L}=0\bigr\}\ra\bigl(\u(n)\op\C^n\bigr){}^*$ by
\begin{equation}
\Psi^t_{p,\up}(f):x\longmapsto \ban{h_{p,\up}^t(x),f}{}_{L^2}
\label{hs6eq5}
\end{equation}
for $f\in\Ker{\cal L}$ with $\int_Lf\,\d V_{g_0\vert_L}=0$ and
$x\in\u(n)\op\C^n$, using the identification
$T_{(p,\up)}U\cong\u(n)\op\C^n$. From above, if $x\in\g$ then
$h_{p,\up}^t(x)\equiv 0$, so $\Psi^t_{p,\up}(f)\vert_{\mathfrak
g}=0$, and $\Psi^t_{p,\up}(f)\in\g^\ci$. Thus $\Psi^t_{p,\up}$ is a
linear map $\bigl\{f\in\Ker{\cal L}:\int_Lf\,\d
V_{g_0\vert_L}=0\bigr\}\ra\g^\ci$. Equations
\eq{hs6eq4}--\eq{hs6eq5} and $H^t(p,\up)=P^t_{p,\up}(f^t_{p,\up})$
imply that $\d K^t\vert_{(p,\up)}=\Psi^t_{p,\up}\ci H^t(p,\up)$, as
we want. Clearly $\Psi^t_{p,\up}$ depends smoothly on~$p,\up$.

It remains to show that $\Psi^t_{p,\up}$ is an isomorphism, with
$\Psi^t_{p,\up}\approx\psi^t$. We claim that for small $t$ and all
$(p,\up)\in U$ we have
\begin{equation}
h_{p,\up}^t(x) \approx t^2\cdot\pd_x\bigl[f^t_{p,\up}\bigr]+\xi^t(x)
\label{hs6eq6}
\end{equation}
in $C^\iy(L)$, where $\pd_x\bigl[f^t_{p,\up}\bigr]$ is the
derivative of the function $(p,\up)\mapsto f^t_{p,\up}$ in direction
$x\in T_{(p,\up)}U$ at $(p,\up)\in U$. To see this, note that
$h_{p,\up}^t(x)$ measures the variation of the family of Hamiltonian
equivalent Lagrangians $(p,\up)\mapsto\Up_{p,\up}\ci t\ci\Phi
(\Ga_{\d f^t_{p,\up}})$ in direction $x$ in $T_{(p,\up)}U$. By the
chain and product rules, we can write this variation as the sum of
two contributions: (a) that due to varying $f^t_{p,\up}$ as a
function of $(p,\up)$ in direction $x$, with $\Up_{p,\up}$ fixed;
and (b) that due to varying $\Up_{p,\up}$ as a function of $(p,\up)$
in direction $x$, with $f^t_{p,\up}$ fixed.

The contributions of type (a) are
$t^2\cdot\pd_x\bigl[f^t_{p,\up}\bigr]$. This is because if we fix
$\Up_{p,\up}$ and vary $f^t_{p,\up}$ in direction $x$ in the
Lagrangian $\Up_{p,\up}\ci t\ci\Phi(\Ga_{\d f^t_{p,\up}})$, then
$\Ga_{\d f^t_{p,\up}}$ in $(T,\hat\om)$ changes by a Hamiltonian
variation from the function $\pd_x\bigl[f^t_{p,\up}\bigr]$; so
$\Phi(\Ga_{\d f^t_{p,\up}})$ in $(B_R,\om_0)$ changes by a
Hamiltonian variation from $\pd_x\bigl[f^t_{p,\up}\bigr]$, as
$\Phi^*(\om_0)=\hat\om$; so $t\ci\Phi (\Ga_{\d f^t_{p,\up}})$ in
$(B_{tR},\om_0)$ changes by a Hamiltonian variation from $t^2\cdot
\pd_x\bigl[f^t_{p,\up}\bigr]$, since $t^*(\om_0)=t^2\cdot\om_0$; so
$\Up_{p,\up}\ci t\ci\Phi (\Ga_{\d f^t_{p,\up}})$ in $(M,\om)$
changes by a Hamiltonian variation from $t^2\cdot
\pd_x\bigl[f^t_{p,\up}\bigr]$, as~$\Up_{p,\up}^*(\om)=\om_0$.

To understand the contributions of type (b), consider the smooth
family of embeddings $\Up_{p,\up}:B_\ep\ra M$ for $(p,\up)\in U$.
The derivative $\pd_x\Up_{p,\up}$ in direction $x\in T_{(p,\up)}U$
at $(p,\up)$ in $U$ is a section of the vector bundle $\Up_{p,\up}^*
(TM)\ra B_\ep$. But $\d\Up_{p,\up}:TB_\ep\ra\Up_{p,\up}^*(TM)$ is an
isomorphism, so $(\d\Up_{p,\up})^{-1}(\pd_x\Up_{p,\up})$ is a vector
field on $B_\ep$. Since $\Up_{p,\up}^*(\om)\equiv\om_0$ for
$(p,\up)\in U$, and $B_\ep$ is simply-connected, this is a
Hamiltonian vector field on $(B_\ep,\om_0)$, so there exists a
smooth function $N_{p,\up}^x:B_\ep\ra\R$, unique up to addition of
constants, such that $(\d\Up_{p,\up})^{-1}(\pd_x\Up_{p,\up})
\cdot\om_0\equiv \d N_{p,\up}^x$ in 1-forms on $B_\ep$. Following
the definitions through shows that the contributions of type (b),
from varying $\Up_{p,\up}$ with $f^t_{p,\up}$ fixed, are
\begin{equation}
q\longmapsto N_{p,\up}^x\ci t\ci\Phi(q,\d f^t_{p,\up} \vert_q)+c,
\label{hs6eq7}
\end{equation}
for $q\in L$, where $c\in\R$ is such that \eq{hs6eq7} integrates to
zero over~$L$.

Now Proposition \ref{hs3prop1}(i) says that $\Up_{p,\up}(0)=p$ and
$\d\Up_{p,\up}\vert_0=\up$. It follows that near 0 in $B_\ep$, the
vector field $(\d\Up_{p,\up})^{-1}(\pd_x\Up_{p,\up})$ on $B_\ep$
approximates the $\u(n)\op\C^n$ vector field $v_x$ on $\C^n$
corresponding to $x$ under the identification
$T_{(p,\up)}U\cong\u(n)\op\C^n$, and thus $N_{p,\up}^x\approx\mu_x$
near 0 in $\C^n$, where $\mu_x:\C^n\ra\R$ is a moment map for $v_x$,
and is unique up to the addition of constants. Since $t\ci\Phi(q,\d
f^t_{p,\up}\vert_q)$ lies in $B_{tR}$, for small $t$ we see that
\eq{hs6eq7} approximates $q\mapsto\mu_x\ci t\ci\Phi(q,\d
f^t_{p,\up}\vert_q)+c$. Also, the proof of Theorem \ref{hs5thm}
shows that when $t$ is small $\nm{f^t_{p,\up}}_{C^{4,\ga}}$ is
small, so we can approximate $f^t_{p,\up}$ by zero, and \eq{hs6eq7}
approximates $\mu_x\ci t\ci\id_L$. The definition of $\xi^t$ above
now implies that \eq{hs6eq7} approximates $\xi^t(x)$, which
proves~\eq{hs6eq6}.

Since $f^t_{p,\up}\perp\Ker{\cal L}$ for all $(p,\up)\in U$, it
follows that $\pd_x\bigl[f^t_{p,\up}\bigr]\perp\Ker{\cal L}$. Thus,
substituting \eq{hs6eq6} into \eq{hs6eq5} and noting that
$f\in\Ker{\cal L}$ shows that
\begin{equation*}
\Psi^t_{p,\up}(f)(x)\approx \ban{\xi^t(x),f}{}_{L^2}
\end{equation*}
for small $t$, for all $f\in\Ker{\cal L}$ with $\int_Lf\,\d
V_{g_0\vert_L}=0$ and $x\in\u(n)\op\C^n$. Since $\psi^t$ is the dual
map to $\xi^t$, using the $L^2$ inner product to identify
$\bigl\{f\in\Ker{\cal L}:\int_Lf\,\d V_{g_0\vert_L}=0\bigr\}$ with
its dual, it follows that $\Psi^t_{p,\up}\approx\psi^t$ for small
$t$. But $\psi^t$ is an isomorphism, which is an open condition, so
$\Psi^t_{p,\up}$ is also an isomorphism for small $t$. By
compactness of $U$, for small enough $t$ these hold uniformly for
all~$(p,\up)\in U$.
\end{proof}

We can now complete the proof of Theorem A. As $\Psi^t_{p,\up}$ is
an isomorphism, $\d K^t\vert_{(p,\up)}=\Psi^t_{p,\up}\ci H^t(p,\up)$
implies that $\d K^t\vert_{(p,\up)}=0$ if and only if
$P^t_{p,\up}(f^t_{p,\up})=H^t(p,\up)=0$. But $L_{p,\up}^t$ is a
Hamiltonian stationary Lagrangian in $(M,\om)$ if and only if
$P^t_{p,\up}(f^t_{p,\up})=0$. Hence for small $t$, $L_{p,\up}^t$ is
a Hamiltonian stationary Lagrangian if and only if $(p,\up)$ is a
stationary point of $K^t:U\ra\R$. But $U$ is a compact manifold
without boundary and $K^t$ is a smooth function, so $K^t$ must have
at least one stationary point $(p,\up)$ in $U$. Then
$L'=L_{p,\up}^t$ satisfies the first part of Theorem~A.

For the second part, suppose also that $L$ is Hamiltonian stable.
Take $(p,\up)$ to be a local minimum of $K^t$ on $U$, which must
exist as $U$ is compact. Then $L'=L_{p,\up}^t$ is Hamiltonian
stationary, as above. We claim that $L'$ is also Hamiltonian stable.
To see this, for $f\in C^\iy(L)$ and small $s\in\R$ write
\begin{equation*}
\Vol_g\bigl(\Up_{p,\up}\!\ci\!t\!\ci\!\Phi(\Ga_{\d f^t_{p,\up}+s\d
f})\bigr)\!=\!\Vol_g\bigl(\Up_{p,\up}\!\ci\!t\!\ci\!\Phi(\Ga_{\d
f^t_{p,\up}})\bigr)\!+\!s^2Q^t_{p,\up}(f)\!+\!O(\md{s}^3),
\end{equation*}
where the homogeneous quadratic form $Q^t_{p,\up}:C^\iy(L)\ra\R$ is
the second variation of $\Vol_g$ at $L'$. We must show that $Q(f)\ge
0$ for all $f\in C^\iy(L)$.

Divide Hamiltonian variations of $L'$ into two kinds: (i) those
coming from functions $f\in(\Ker{\cal L})^\perp$, that is, $f\in
C^\iy(L)$ with $f\perp\Ker{\cal L}$; and (ii) those coming from the
family of Lagrangians $L_{p',\up'}^t$ in $(M,\om)$, for $(p',\up')$
in $U$ close to $(p,\up)$. The vector space of functions $f\in
C^\iy(L)$ corresponding to Hamiltonian variations of type (ii) turns
out to be $W^t_{p,\up}=\bigl\{h_{p,\up}^t(x):x\in
T_{(p,\up)}U\bigr\}\op\an{1}$, for $h_{p,\up}^t$ as in the proof of
Proposition \ref{hs6prop}. We have $C^\iy(L)=(\Ker{\cal L})^\perp\op
W^t_{p,\up}$.

For Hamiltonian variations of type (i), the second variation of the
volume functional $\Vol_g$ at $L_{p,\up}^t$ approximates the second
variation of the volume functional $\Vol_{g_0}$ at $tL$ in $\C^n$ on
functions $f\in (\Ker{\cal L})^\perp$, or equivalently, the second
variation of $t^n\cdot\Vol_{g_0}$ at $L$ in $\C^n$ on functions
$f\in (\Ker{\cal L})^\perp$. Using this we can show that
\begin{equation}
Q^t_{p,\up}(f)=t^n\an{f,{\cal L}f}_{L^2} +O\bigl(t^{\ga}\nm{f}_{
L^2_2}^2\bigr)\quad\text{for all $f\in(\Ker{\cal L})^\perp$,}
\label{hs6eq8}
\end{equation}
for some $\ga>n$. As the second variation of $\Vol_{g_0}$ at $L$ is
nonnegative, $\cal L$ is a nonnegative fourth-order linear elliptic
operator on a compact manifold $L$. Using this we can show that
there exists $\la>0$ such that $\an{f,{\cal L}f}_{L^2}\ge
\la\nm{f}_{L^2_2}^2$ for all $f\in(\Ker{\cal L})^\perp$. Thus the
term $t^n\an{f,{\cal L}f}_{L^2}$ in \eq{hs6eq8} dominates the term
$O\bigl(t^{\ga}\nm{f}_{L^2_2}^2\bigr)$ for small $t$, as $\ga>n$.
Hence the second variation of $\Vol_g$ on variations of type (i) is
positive definite, for small~$t$.

For variations of type (ii), the second variation of $\Vol_g$ is the
second variation of the function $(p',\up')\mapsto
\Vol_g(L_{p',\up'}^t)=t^n\cdot K^t(p',\up')$ at $(p,\up)$. But
$(p,\up)$ is a local minimum of $K^t$ at $(p,\up)$, so this second
variation is nonnegative. Therefore $Q^t_{p,\up}(f)\ge 0$ for
all~$f\in W^t_{p,\up}$.

A general variation $f$ may be written uniquely as $f=f_1+f_2$ for
$f_1\in(\Ker{\cal L})^\perp$ of type (i) and $f_2\in W^t_{p,\up}$ of
type (ii). We claim that $Q^t_{p,\up}(f)=Q^t_{p,\up}(f_1)
+Q^t_{p,\up}(f_2)$, that is, the bilinear terms in $f_1\ot f_2$ in
$Q^t_{p,\up}(f_1+f_2)$ are zero. This holds because the definition
of the family of Lagrangians $L_{p',\up'}^t$ with
$P^t_{p',\up'}(f^t_{p',\up'})\in\Ker{\cal L}$ means that the volume
of $L_{p',\up'}^t$ is stationary under Hamiltonian variations coming
from $f_1\in(\Ker{\cal L})^\perp$ not just at the single point
$(p,\up)$, but for all $(p',\up')\in U$. That is, if
$f_1\in(\Ker{\cal L})^\perp$ then
\begin{equation*}
\ts\frac{\d}{\d s}\bigl[
\Vol_g\bigl(\Up_{p',\up'}\!\ci\!t\!\ci\!\Phi(\Ga_{\d
f^t_{p',\up'}+s\d f_1})\bigr)\bigr]\big\vert_{s=0}=0.
\end{equation*}
Differentiating this identity at $(p',\up')=(p,\up)$ in the
direction in $T_{(p,\up)}U$ induced by the Hamiltonian variation
$f_2\in W^t_{p,\up}$ implies that the $f_1\ot f_2$ term in
$Q^t_{p,\up}(f_1+f_2)$ is zero. Therefore $Q^t_{p,\up}(f)=
Q^t_{p,\up}(f_1)+Q^t_{p,\up}(f_2)\ge 0$, since $Q^t_{p,\up}(f_1)\ge
0$ and $Q^t_{p,\up}(f_2)\ge 0$. So $L'$ is Hamiltonian stable. This
completes the proof.

\section{Conclusions}
\label{hs7}

We finish with a question about the family of Hamiltonian stationary
Lagrangians $L'$ in a fixed Hamiltonian isotopy class $\cal HI$ in a
compact symplectic manifold $(M,\om)$. Although in
\S\ref{hs4}--\S\ref{hs6} we dealt with a family of Lagrangians
$L_{p,\up}^t$ parametrized by $(p,\up)\in U$, as the whole
construction is $G$-equivariant we can think of this family of
Lagrangians as parametrized by $(p,\up)G$ in $U/G$. The
$G$-invariant function $K^t:U\ra\R$ descends to $K^t_G:U/G\ra\R$. In
\S\ref{hs6} we proved that $L_{p,\up}^t$ is a Hamiltonian stationary
Lagrangian if and only if $(p,\up)G$ is a stationary point of
$K^t_G$ on the compact manifold~$U/G$.

Suppose the only Hamiltonian stationary Lagrangians $L'$ in this
Hamiltonian isotopy class $\cal HI$ are of the form $L_{p,\up}^t$.
Then we would have identified the family of Hamiltonian stationary
Lagrangians $L'$ in $\cal HI$, which is the critical locus of a real
function on an infinite-dimensional, noncompact manifold $\cal HI$,
with the critical locus of a real function on a finite-dimensional,
compact manifold~$U/G$.

This is suggestive. There are several important areas in geometry,
dealing either with counting invariants such as Donaldson,
Gromov--Witten, or Donaldson--Thomas invariants, or with Floer
homology theories, for which the original motivation comes from
considering some infinite-dimensional, noncompact moduli space $\cal
M$ of connections or submanifolds, and then treating $\cal M$ as if
it were a finite-dimensional compact manifold.

If $Y$ is a compact manifold and $f:Y\ra\R$ is a Morse function,
then the number of critical points of $f$, counted with signs, is
$\chi(Y)$, and using the gradient flow lines of $f$ between critical
points one can construct the (Morse) homology $H_*(Y;\R)$. The
invariants and homology theories mentioned above work by counting
critical points or gradient flow lines of a functional $F:{\cal
M}\ra\R$ on an infinite-dimensional, noncompact manifold $\cal M$;
the answers turn out to be independent of most of the geometric
choices in the definition of $F$, even though $\cal M$ is neither
finite-dimensional nor compact. This motivates the following:

\begin{quest} Let\/ $(M,\om)$ be a compact symplectic manifold, $g$ a
Riemannian metric on\/ $M$ compatible with\/ $\om,$ and\/ ${\cal
HI}$ a Hamiltonian isotopy class of compact Lagrangians\/ $L$ in\/
$(M,\om)$. Write\/ $\Vol_g:{\cal HI}\ra\R$ for the volume
functional.

Can one define some invariant\/ $I({\cal HI})\in\Z$ which `counts'
(with multiplicity and sign) Hamiltonian stationary Lagrangians\/
$L$ in ${\cal HI},$ that is, stationary points of\/ $\Vol_g,$ and
gives an answer independent of the choice of\/~$g$?

Can one define some kind of Floer homology theory\/ $HF_*({\cal
HI})$ by studying the gradient flow of\/ $\Vol_g$ between critical
points, whose Euler characteristic is\/ $I({\cal HI}),$ and which is
independent of\/ $g$ up to canonical isomorphism?
\label{hs7quest}
\end{quest}

In the case of Theorem A, since the family of Hamiltonian stationary
Lagrangians $L'$ we have constructed corresponds to the critical
points of a function $K^t:U/G\ra\R$, we would expect the answers
$I({\cal HI})=\chi(U/G)$ and $HF_*({\cal HI})\cong H_*(U/G;\R)$.
Since $U/G$ is a fibre bundle over $M$ with fibre $\U(n)/G$ we have
$\chi(U/G)=\chi\bigl(\U(n)/G\bigr)\chi(M)$. If $L$ is
$T^n_{a_1,\ldots,a_n}$ in \eq{hs2eq11} with $a_1,\ldots,a_n>0$
distinct, then $G$ is the maximal torus $T^n$ in $\U(n)$, and
$\chi\bigl(\U(n)/G\bigr)=n!$, so we expect~$I({\cal
HI})=n!\,\chi(M)$.
\bigskip

{\parindent 0pt\parskip 3pt
{\scshape The Mathematical Institute, 24-29 St.~Giles,
Oxford, OX1 3LB, U.K.}

{\it E-mail:} {\tt joyce@maths.ox.ac.uk}
\medskip

{\scshape Department of Mathematics and Taida Institute
of Mathematical Sciences, National Taiwan University, Taipei 10617,
Taiwan}

{\it E-mail:} {\tt yilee@math.ntu.edu.tw}
\medskip

{\scshape Department of Mathematics, Stanford University,
Stanford, CA 94305-2125, U.S.A.}

{\it E-mail:} {\tt schoen@math.stanford.edu}}

\end{document}